\documentclass[12pt]{article}

\title{Axiomatic $S^1$ Morse-Bott theory}
\author{Michael Hutchings\footnote{Partially supported by NSF grant DMS-1406312 and a Simons Fellowship.}\;  and Jo Nelson\footnote{Partially supported by NSF grant DMS-1303903.}\;}
\date{}
\usepackage{amssymb}
\usepackage{latexsym}
\usepackage{amsmath}
\usepackage{amsthm}
\usepackage{relsize}
\usepackage[matha,  mathx]{mathabx}
\usepackage{marvosym}
\usepackage{wasysym}
\usepackage{scalerel}

\usepackage{phonetic}

\usepackage[mathscr]{euscript}

\usepackage{amscd}
\usepackage{overpic}
\usepackage{enumerate}
\usepackage{color}
\usepackage[usenames,dvipsnames,svgnames,table]{xcolor}
\usepackage{graphicx}

\usepackage{hyperref}
\hypersetup{colorlinks}

\definecolor{indigo}{RGB}{51,0,102}
\definecolor{brightpurple}{RGB}{102,0,153}
\definecolor{fuchsia}{RGB}{180,51,180}
\definecolor{jolightpurple}{RGB}{188,171,240}

\hypersetup{colorlinks,
linkcolor=brightpurple,
filecolor=brightpurple,
urlcolor=indigo,
citecolor=fuchsia}

\usepackage[margin=1.25in]{geometry}

\normalbaselines
\newcommand{\mc}[1]{{\mathcal #1}}

\numberwithin{equation}{section}

\newtheorem{theorem}{Theorem}[section]
\newtheorem{proposition}[theorem]{Proposition}

\newtheorem{lemma}[theorem]{Lemma}
\newtheorem{lemma-definition}[theorem]{Lemma-Definition}

\theoremstyle{definition}
\newtheorem{definition}[theorem]{Definition}
\newtheorem{remark}[theorem]{Remark}

\newtheorem{example}[theorem]{Example}
\newtheorem{convention}[theorem]{Convention}

\newcommand{\R}{{\mathbb R}}
\newcommand{\N}{{\mathbb N}}
\newcommand{\Z}{{\mathbb Z}}

\newcommand{\op}{\operatorname}

\newcommand{\M}{\mc{M}}

\newcommand{\Ker}{\op{Ker}}

\newcommand{\tensor}{\otimes}

\newcommand{\bpm}{\begin{pmatrix}}
\newcommand{\epm}{\end{pmatrix}}

\newcommand{\J}{\mathbb{J}}

\newcommand{\ca}{{\mbox{\lightning}}}

\begin{document}

\maketitle

\begin{abstract}
In various situations in Floer theory, one extracts homological invariants from ``Morse-Bott'' data in which the ``critical set'' is a union of manifolds, and the moduli spaces of ``flow lines'' have evaluation maps taking values in the critical set. This requires a mix of analytic arguments (establishing properties of the moduli spaces and evaluation maps) and formal arguments (defining or computing invariants from the analytic data). The goal of this paper is to isolate the formal arguments, in the case when the critical set is a union of circles. Namely, we state axioms for moduli spaces and evaluation maps (encoding a minimal amount of analytical information that one needs to verify in any given Floer-theoretic situation), and using these axioms we define homological invariants. More precisely, we define a (almost) category of ``Morse-Bott systems''. We construct a ``cascade homology'' functor on this category, based on ideas of Bourgeois and Frauenfelder, which is ``homotopy invariant''.  This machinery is used in our work on cylindrical contact homology.
\end{abstract}

\tableofcontents

\bigskip

%%%%%%%%%%%%%%%%%%%%%%%%%%%%%%
%Here are more massive $\circlearrowleft$.
%$\ \mathlarger{\mathlarger{\mathlarger{\circlearrowleft}}} \ \mathlarger{\mathlarger{\mathlarger{\mathlarger{\mathlarger{\mathlarger{\circlearrowleft}}}}}}$ Here is how to make stuff appear under the $\circlearrowleft$ \[ \underset{p_{\overline{\gamma_1}}}{\circlearrowleft} \] or  \[ \underset{p_{\overline{\gamma_1}}}{\mathlarger{\mathlarger{\mathlarger{\circlearrowleft}}}} \]

%Here are some macros I made for coloring text:  \tr{now the text is red!} \tg{green text} \tv{violet text} \tp{purple text} \tb{blue text} \\

\section{Introduction}

There are now many versions of Floer theory, which are used to define topological invariants of various kinds of objects, such as symplectomorphisms, pairs of Lagrangian submanifolds of a symplectic manifold, contact manifolds, smooth three-manifolds, etc. In the most basic versions of Floer theory, given an object, usually together with a generic choice of certain auxiliary data, one obtains a discrete set of ``critical points'', and for each pair of critical points a moduli space of ``flow lines'' between them. The invariant is then obtained as the homology of a chain complex which is generated by the critical points and whose differential counts flow lines, analogously to classical Morse homology. The proof that the differential has square zero involves gluing two flow lines when the lower limit of the first flow line and the upper limit of the second flow line are at the same critical point.

In some less well-behaved Floer theoretic situations, instead of a discrete set of critical points, one obtains a union of ``critical submanifolds'', analogous to the critical set of a Morse-Bott function on a finite-dimensional manifold. In this case, given two critical submanifolds, there is still a moduli space of flow lines between them. Now there are also upper and lower evaluation maps from the moduli space of flow lines to the two critical submanifolds (or more generally, certain manifolds associated to them). Two flow lines can be glued only if the lower evaluation map on the first flow line agrees with the upper evaluation map on the second flow line. The definition of homological invariants in this situation is more complicated and combines analytical arguments (establishing properties of the moduli spaces and evaluation maps) with formal arguments (extracting invariants from the analytic data).

The goal of this paper is to isolate the formal arguments needed to define homological invariants in such Morse-Bott situations, in the special case when the critical submanifolds (more precisely the manifolds associated to them) are circles. In particular, we state axioms for a ``Morse-Bott system'', and given a Morse-Bott system, we define its ``cascade homology''. We also define a notion of ``morphism'' between Morse-Bott systems, which almost makes Morse-Bott systems into a category. (The reason for the word ``almost'' is that two morphisms, such that the target of the first morphism equals the source of the second, are composable only under certain transversality conditions.) Given a morphism, we define an induced map on cascade homology. Finally, we show that the induced maps are functorial, and invariant under ``homotopies'' of morphisms. The result is a blueprint for defining Floer-theoretic invariants, by analytically establishing various axioms and then invoking the formal machinery of this paper. We now describe this in more detail.

\subsection{Summary of results}

The precise definition of ``Morse-Bott system'' is given in \S\ref{sec:defmbs}. Some key features are the following. A Morse-Bott system includes a set $X$; one can think of each element of $S$ as referring to a ``critical submanifold''. For each $x\in X$, there is an associated closed connected oriented 1-manifold $S(x)$. Given distinct elements $x_+,x_-\in X$, and given an integer $d\in\{0,1,2,3\}$, there is a moduli space $M_d(x_+,x_-)$ of ``flow lines'' from $x_+$ to $x_-$, which is a smooth $d$-dimensional manifold. (In many cases, given $x_+$ and $x_-$, there is only one value of $d$ for which this moduli space can be nonempty. One could also consider moduli spaces for $d>3$, but these are not relevant for our story.) There are smooth ``evaluation maps''
\[
\begin{split}
e_+: M_d(x_+,x_-) &\longrightarrow S(x_+),\\
e_-: M_d(x_+,x_-) &\longrightarrow S(x_-).
\end{split}
\]
Also, for each $x\in X$ there is a local system $\mc{O}_x$ on $S(x)$ locally isomorphic to $\Z$, and there is an orientation of $M_d(x_+,x_-)$ with values in $e_+^*\mc{O}_{x_+}\tensor e_-^*\mc{O}_{x_-}$.

The above data are required to satisfy various axioms. Most importantly, there is a ``compactification'' axiom which asserts that $M_1(x_+,x_-)$, as well as $M_2(x_+,x_-)$ with a generic point constraint on $e_+$ or $e_-$, or $M_3(x_+,x_-)$ with generic point constraints on both $e_+$ and $e_-$, has a compactification to a compact topological $1$-manifold, whose boundary is explicitly described in terms of fiber products of moduli spaces.

To extract a homological invariant out of this structure, we use the ``cascade'' approach. Cascades were introduced by Bourgeois \cite{bourgeois}, and discovered independently\footnote{Some related ideas appeared earlier in \cite{cfhw,pss}. In the work of Bourgeois, the emphasis is on describing, in Morse-Bott terms, what one would obtain after perturbing to a nondegenerate (non-Morse-Bott) situation. By contrast, in the work of Frauenfelder, the idea is to define invariants and prove invariance entirely in the Morse-Bott world. This is closer to our philosophy, since in our main examples of interest in contact homology, there is no apparent way to perturb to a non-Morse-Bott situation. See Example~\ref{ex:nch}.} by Frauenfelder \cite{frauenfelder}. The original idea would be to choose a generic auxiliary Morse function $f_x$ on each manifold $S(x)$, and define a chain complex over $\Z$ which is generated by pairs $(x,p)$ where $p$ is a critical point of $f_x$. The chain complex differential counts cascades, which are alternating sequences of gradient flow lines of the Morse functions $f_x$ and elements of the moduli spaces $M_d(x_+,x_-)$.

In our situation where each $S(x)$ is a circle, we use a streamlined version of this construction, following \cite{bee}, in which one chooses only one base point $p_x$ on each circle $S(x)$. (One can think of this as a limit in which the critical points of $f_x$ all approach $p_x$.) The chain complex has two generators $\widehat{x}$ and $\widecheck{x}$ for each $x\in X$. (One can think of these as a maximum and minimum respectively of $f_x$, which have collided at $p_x$.) If $x_+$ and $x_-$ are distinct elements of $X$, a cascade from $\widehat{x}_+$ or $\widecheck{x}_+$ to $\widehat{x}_+$ or $\widehat{x}_-$ consists of a sequence $(u_1,\ldots,u_k)$ for some positive integer $k$, such that there are distinct $x_0,\ldots,x_k\in X$ with $x_+=x_0$, $x_-=x_k$, and $u_i\in M_{d_i}(x_{i-1},x_i)$. For $i=1,\ldots,k-1$, we require that the points $p_{x_i}$, $e_-(u_i)$, and $e_+(u_{i+1})$ on $S(x_i)$ are distinct and positively cyclically ordered with respect to the orientation of $S(x_i)$. If we are starting from $\widecheck{x}_+$, then we also impose the point constraint $e_+(u_1)=p_{x_+}$; and if we are ending at $\widehat{x}_-$, then we also impose the point constraint $e_-(u_k)=p_{x_-}$. The differential coefficient counts such cascades where the total moduli space dimension is the number of point constraints. When $x_+=x_-$, all differential coefficients are defined to be zero, except that our orientation conventions in \S\ref{sec:cmsd} require that the differential coefficient
\begin{equation}
\label{eqn:minustwo}
\big\langle \partial\widehat{x},\widecheck{x}\big\rangle = \left\{\begin{array}{cl} 0, & \mbox{if $\mc{O}_x$ is trivial},\\ -2, & \mbox{if $\mc{O}_x$ is nontrivial}. \end{array}\right.
\end{equation}

The results in this paper can now be summarized as follows.

\begin{theorem}
\label{thm:main}
\begin{description}
\item{(a)} Let $A$ be a Morse-Bott system (see Definition~\ref{def:mbs}). Then the cascade homology $H_*^\ca(A)$ (see Definition~\ref{def:hca}) is well-defined, independently of the choice of base points.
\item{(b)} Let $\Phi$ be a morphism of Morse-Bott systems from $A_1$ to $A_2$ (see Definition~\ref{def:morphism}). Then:
\begin{description}
\item{(i)}
The induced map on cascade homology
\[
\Phi_*: H_*^\ca(A_1) \longrightarrow H_*^\ca(A_2)
\]
(see Definition~\ref{def:phistar}) is well-defined independently of choices.
\item{(ii)} If $A_1=A_2$ and $\Phi$ is the identity morphism (see Example~\ref{ex:identity}), then $\Phi_*$ is the identity map on cascade homology.
\item{(iii)} If $\Psi$ is a morphism from $A_2$ to $A_3$, and if $\Phi$ and $\Psi$ are composable (see Definition~\ref{def:composable}), then the composition $\Psi\circ\Phi$ (see Definition~\ref{def:composition}) satisfies
\[
(\Psi\circ\Phi)_* = \Psi_* \circ \Phi_*: H_*^\ca(A_1) \longrightarrow H_*^\ca(A_3).
\]
\item{(iv)}
If $\Phi'$ is another morphism from $A_1$ to $A_2$ which is homotopic to $\Phi$ (see Definition~\ref{def:homotopy}), then
\[
\Phi_* = (\Phi')_*: H_*^\ca(A_1) \longrightarrow H_*^\ca(A_2).
\]
\end{description}
\end{description}
\end{theorem}

To use this theorem to define Floer-theoretic invariants of some class of objects, the procedure is as follows: (1) For each object, together with generic auxiliary data if necessary, define a Morse-Bott system. (2) For two different objects (with auxiliary data as needed), define a homotopy class of morphisms between the corresponding Morse-Bott systems. (3) Show that the composition of some morphism in this homotopy class with a morphism going in the other direction is homotopic to the identity.

\subsection{Examples}

The following two examples are the main examples we have in mind and the reason we are writing this paper.

\begin{example}
\label{ex:nch}
Let $Y$ be a closed odd-dimensional manifold. Let $\lambda$ be a contact form on $Y$, let $R$ denote the associated Reeb vector field, and let $\xi=\Ker(\lambda)$ denote the associated contact structure. Assume that $\lambda$ is nondegenerate and {\em hypertight\/} (meaning that there are no contractible Reeb orbits). Let $\J=\{J_t\}_{t\in S^1}$ be a generic $S^1$-family of $\lambda$-compatible\footnote{An almost complex structure $J$ on $\R\times Y$ is {\em $\lambda$-compatible\/} if $J$ sends $\xi$ to itself, such that $J$ is compatible with the linear symplectic form $d\lambda$ on $\xi$; $J$ is invariant under translation of the $\R$ factor; and $J(\partial_s)=R$, where $s$ denotes the $\R$ coordinate.} almost complex structures on $\R\times Y$. In \cite{equi19}, we associate to this data a Morse-Bott system $A(Y,\lambda;\J)$ where:
\begin{itemize}
\item
$X$ is the set of (not necessarily simple) Reeb orbits.
\item
If $x$ is a Reeb orbit, then:
\begin{itemize}
\item
$S(x)$ is the image of $x$ in $Y$, oriented via the Reeb vector field.
\item
The local system $\mc{O}_x$ comes from the theory of coherent orientations \cite{bm,fh}, and is trivial if and only if $x$ is a good\footnote{As in \cite{egh}, a Reeb orbit $x$ is {\em good\/} if $x$ is not an even-degree multiple cover of a Reeb orbit $x'$ for which the Conley Zehnder indices of $x$ and $x'$ have opposite parity.} Reeb orbit.
\item
The grading $|x|$ (see Definition~\ref{def:mbs}) is the parity of the Conley-Zehnder index of $x$.
\end{itemize}
\item
If $x_+$ and $x_-$ are distinct Reeb orbits, let $\M^\J(x_+,x_-)$ denote the moduli space of $\J$-holomorphic cylinders from $x_+$ to $x_-$, i.e.\ the set of maps $u:\R\times S^1\to \R\times Y$ satisfying the equations
\begin{gather*}
\partial_s u + J_t\partial_tu=0,\\
\lim_{s\to\pm\infty}\pi_\R u(s,\cdot) = \pm\infty,\\
\lim_{s\to\pm\infty}\pi_Yu(s,\cdot) \mbox{ is a parametrization of $x_\pm$,}
\end{gather*}
modulo $\R$ translation in the domain. Let $\M_d^\J(x_+,x_-)$ denote the set of elements of $M^\J(x_+,x_-)$ with Fredholm index $d$. We then have
\[
M_d(x_+,x_-)=\M^\J_{d+1}(x_+,x_-)/\R,
\]
where $\R$ acts by translation of the $\R$ factor in the target. The compactifications of these moduli spaces are defined by adjoining ``broken holomorphic cylinders''.
\item
The evaluation map $e_\pm$ sends $u\mapsto \lim_{s\to\pm\infty}\pi_Y(u(s,0))$.
\end{itemize}

Analytic arguments in \cite{equi19} show that $A(Y,\lambda;\J)$ satisfies the axioms of a Morse-Bott system. It then follows from Theorem~\ref{thm:main}(a) that the Morse-Bott system $A(Y,\lambda;\J)$ has a well-defined cascade homology. This cascade homology is the {\em nonequivariant contact homology\/} of $(Y,\lambda,\J)$, which we denote by $NCH_*(Y,\lambda;\J)$.

To prove that nonequivariant contact homology\footnote{Nonequivariant contact homology is a contact analogue of the (Morse-Bott) Floer theory for autonomous Hamiltonians studied by Bourgeois-Oancea \cite{bo-duke}. The paper \cite{bo-duke} identified this Morse-Bott Floer theory with the Floer theory for a (non-Morse-Bott) nonautonomous perturbation of the Hamiltonian. In our contact situation we cannot make an analogous non-Morse-Bott perturbation, so if we want to prove that nonequivariant contact homology is a topological invariant, we need to work entirely within the Morse-Bott world.}
 depends only on the contact structure, we show in \cite{equi19} that if $\lambda'$ is another hypertight contact form with $\Ker(\lambda')=\xi$, and if $\J'$ is a generic $S^1$-family of $\lambda'$ almost complex structures, then there is a morphism of Morse-Bott systems (obtained by counting holomorphic cylinders in a completed symplectic cobordism) from $A(Y,\lambda;\J)$ to $A(Y,\lambda';\J')$, which is well-defined up to homotopy of Morse-Bott systems. Thus by Theorem~\ref{thm:main}(b), we obtain a canonical map
\begin{equation}
\label{eqn:NCHiso}
NCH_*(Y,\lambda;\J) \longrightarrow NCH_*(Y',\lambda';\J').
\end{equation}
Finally, we show in \cite{equi19} that the composition of one of the morphisms from $A(Y,\lambda;\J)$ to $A(Y,\lambda';\J')$ with one of the morphisms going in the other direction is homotopic to the identity. It then follows from Theorem~\ref{thm:main}(b) that the map \eqref{eqn:NCHiso} is an isomorphism. We conclude in \cite{equi19} that {\em nonequivariant contact homology is an invariant of closed contact manifolds $(Y,\xi)$ that admit nondegenerate hypertight contact forms\/}.
\end{example}

\begin{example}
If $Y$ is a closed manifold and $\xi$ is a contact structure on $Y$ which admits a nondegenerate hypertight contact form, a variant of the above construction is used in \cite{equi19} to define the {\em $S^1$-equivariant contact homology\/} $CH_*^{S^1}(Y,\xi)$. Again, the analysis in \cite{equi19} produces Morse-Bott systems, morphisms, and homotopies, and then Theorem~\ref{thm:main} gives an invariant (of closed contact manifolds that admit nondegenerate hypertight contact forms).
\end{example}

\begin{example}
\label{ex:morse}
A more classical example arises when $Z$ is a closed smooth manifold, $f:Z\to\R$ is a Morse-Bott function whose critical set is a union of $1$-manifolds, and $g$ is a generic metric on $Z$. One can then define a Morse-Bott system where:
\begin{itemize}
\item
$X$ is the set of components of the critical set of $f$.
\item
If $x\in X$, then:
\begin{itemize}
\item
$S(x)$ is the component $x$, with an arbitrary orientation.
\item
The local system $\mc{O}_x$ is the orientation bundle of the bundle of unstable manifolds of the critical points in $S(x)$. That is, if $p\in S(x)$, and if $\mc{D}(p)$ denotes the unstable manifold of $p$, then 
\[
\mc{O}_x(p)=H_{\op{ind}(x)}(\mc{D}(p),\mc{D}(p)\setminus\{p\}).
\]
Here $\op{ind}(x) = \op{dim}(\mc{D}(p))$ denotes the (lower) Morse index of the component $x$.
\end{itemize}
\item Let $x_+$ and $x_-$ be distinct components of the critical set. Then $M_d(x_+,x_-)$ is nonempty only if $d=\op{ind}(x_+) - \op{ind}(x_-)$. In this case, $M_d(x_+,x_-)$ is the moduli space of maps $\gamma:\R\to Z$ satisfying the equations
\begin{gather*}
\gamma'(s) = \nabla f(\gamma(s)),\\
\lim_{s\to\pm\infty}\gamma(s) \in S(x_\pm).
\end{gather*}
Here we mod out the set of maps $\gamma$ by $\R$ translation in the domain. The compactifications of these moduli spaces are defined by adjoining ``broken flow lines''.
\item The evaluation map $e_\pm$ sends $\gamma\mapsto \lim_{s\pm\infty}\gamma(s)$.
\end{itemize}
The cascade homology of this Morse-Bott system is canonically isomorphic to the singular homology of the manifold $Z$. Indeed, following \cite{bourgeois}, one can perturb the Morse-Bott function $f$ to a Morse function $f'$, such that the cascade chain complex is canonically isomorphic at the chain level to the Morse complex of $(f',g)$, with each component of the critical set of $f$ contributing two critical points of $f'$, both close to the base point used to define the cascade chain complex. See e.g.\ \cite{bh2}.
\end{example}

\subsection{Comparison with other approaches}

\begin{remark}
Zhengyi Zhou \cite{zhou} has independently developed an abstract Morse-Bott theory which is similar in spirit to what we are doing here, but applicable in different situations. He defines a ``flow category'', after work of Cohen-Jones-Segal \cite{cjs}, which is related to our notion of ``Morse-Bott system''. In a flow category, the ``critical submanifolds'' can have arbitrary dimension, unlike the Morse-Bott systems in this paper which only have one-dimensional critical submanifolds. However a flow category is also required to satisfy strong analytic assumptions, in particular that the moduli spaces have compactifications which are smooth manifolds with corners; while we make weaker analytic assumptions, in which the only compactifications that arise are topological $1$-manifolds with boundary. Zhou defines a kind of Morse-Bott cohomology out a flow category using de Rham theory, and in particular with coefficients in $\R$. One can presumably also set up cascade homology over $\Z$ in this setting.
\end{remark}

\begin{remark}
There is also an older approach to Morse-Bott theory due to Fukaya \cite{fukaya}. (See \cite{bh1} for a variant of this for Morse-Bott functions on finite-dimensional manifolds.) The idea is to define a chain complex which consists of appropriate chains in the (manifolds associated to the) critical submanifolds. The differential is the sum of the usual boundary operator on chains, plus a term which consists of a pullback-pushforward of chains over the moduli spaces. This approach has the nice feature that it does not involve any choice of base points. One can implement this theory for Morse-Bott systems and prove that it is canonically isomorphic to cascade homology. However we have omitted this story in order to keep this paper to a reasonable length.
\end{remark}

\subsection{The rest of the paper}

In \S\ref{sec:mbs}, we define the notions of Morse-Bott system, morphism of Morse-Bott systems, composition of morphisms, and homotopy of morphisms. We also prove that the composition of morphisms is a morphism. We have endeavoured to make a minimum of assumptions, with the result that the definitions are somewhat long. In many ``real-life'' situations, one knows stronger transversality and compactification properties which are simpler to state. See the remarks in \S\ref{sec:defmbs}.

In \S\ref{sec:cascade} we set up cascade moduli spaces and prove their key properties. We use these to define the cascade homology of a Morse-Bott system, as well as a map on cascade homology induced by a morphism of Morse-Bott systems. We prove that the induced maps are functorial, and invariant under homotopy of morphisms. Finally, we show that the above constructions do not depend on the choice of base points. The conclusion in \S\ref{sec:conclusion} reviews where all of the points in Theorem~\ref{thm:main} are proved.

\paragraph{Acknowledgments.} We thank Zhengyi Zhou for helpful conversations.

\section{Morse-Bott systems}
\label{sec:mbs}

In this section we give the precise definitions of ``Morse-Bott system'', ``morphism'' of Morse-Bott system, and ``homotopy'' of morphisms. We also define the composition of ``composable'' morphisms and prove that this is a morphism.

\subsection{Conventions}
\label{sec:conventions}

\subsubsection{Orientation of level sets}

If $X$ is an $n$-dimensional oriented manifold, if $S$ is an oriented $1$-manifold, if $f:X\to S$ is a smooth map, and if $p\in S$ is a regular value of $f$, then we orient $f^{-1}(p)$ using the ``derivative first'' convention. This means that if $x\in f^{-1}(p)$, and if $(v_1,\ldots,v_n)$ is an oriented basis for $T_xX$ such that $df_x(v_1)>0$, then $(v_2,\ldots,v_n)$ is an oriented basis for $T_x(f^{-1}(p))$.

\subsubsection{Orientation of fiber products}

Let $X$ and $Y$ be oriented manifolds of dimension $m$ and $n$ respectively, let $S$ be an oriented $1$-manifold, and let $e_-:X\to S$ and $e_+:Y\to S$ be smooth maps. Suppose that the fiber product $X\times_S Y$ is cut out transversely. We then orient this fiber product as follows. Given $(x,y)\in X\times Y$ with $e_-(x)=e_+(y)$, choose $(u_1,v_1)\in T_xX\oplus T_yY$ such that $de_-(u_1)-de_+(v_1)>0$ with respect to the orientation on $S$. Choose $(u_i,v_i)_{i=2,\ldots,m+n}$ with $de_-(u_i)=de_+(v_i)$ such that $(u_1,v_1),\ldots,(u_{m+n},v_{m+n})$ is an oriented basis for $T_x X \oplus T_y Y$. Then $(u_2,v_2),\ldots,(u_{m+n},v_{m+n})$ is an oriented basis for $T_{(x,y)}(X\times_S Y)$ if and only if $m$ is odd.

This convention is chosen so that fiber product is associative.
Namely, if $Z$ is another oriented manifold, if $S'$ is another oriented $1$-manifold, and if $e_-:Y\to S'$ and $e_+:Z\to S'$ are smooth maps, then we have an equality of oriented manifolds
\begin{equation}
\label{eqn:fpna}
X\times_S (Y\times_{S'} Z) = (X\times_S Y) \times_{S'} Z
\end{equation}
whenever all fiber products in this equation are cut out transversely.

Another nice property of this convention is that when $X$ or $Y$ is equal to $S$, with $e_-$ or $e_+$ equal to the identity map, we have
\begin{equation}
\label{eqn:fpid}
\begin{split}
X\times_S S &= X,\\
S\times_S Y &= Y
\end{split}
\end{equation}
as oriented manifolds.

Also note that if $X$ or $Y$ is a positively oriented point $p\in S$, and if $e_-$ or $e_+$ respectively is the inclusion $\{p\}\to S$, then we have
\begin{equation}
\label{eqn:fpls}
\begin{split}
\{p\}\times_S Y &= e_+^{-1}(p)\subset Y,\\
(-1)^{\dim(X)-1}X\times_S\{p\} &= e_-^{-1}(p)\subset X
\end{split}
\end{equation}
as oriented manifolds.

If $X$ and/or $Y$ have boundary, then the boundary (codimension 1 stratum) of the fiber product $X\times_S Y$ is given by
\begin{equation}
\label{eqn:bfp}
\partial\left(X\times_S Y\right) = (\partial X) \times_S Y \cup (-1)^{\dim(X)-1}X\times_S \partial Y.
\end{equation}

\subsubsection{Compactifications}

Let $M$ be a smooth oriented $1$-manifold without boundary. In this paper, we define a ``compactification'' of $M$ to be a compact oriented topological $1$-manifold $\overline{M}$, possibly with boundary, such that $M$ is an open subset of $\overline{M}$, the orientation of $\overline{M}$ restricts to the orientation of $M$, and $\overline{M}\setminus M$ is finite.

Note that if $\overline{M}$ is a compactification of $M$, then $\overline{M}\setminus M$ contains $\partial\overline{M}$, but $\overline{M}\setminus M$ might also contain finitely many additional points. For example, under the above definition, the closed interval $[0,2]$ is a compactification of the union of open intervals $(0,1)\cup(1,2)$. Here $1$ is an ``extra point'' in the compactification which is not in the boundary.  We need to allow such points in order for composition of morphisms of Morse-Bott systems to work; see Proposition~\ref{prop:composition} below.

\subsection{The fundamental definition}
\label{sec:defmbs}

\begin{definition}
\label{def:mbs}
A {\em Morse-Bott system} is a tuple $(X,|\cdot|,S,\mc{O},M_*,e_\pm)$ where:
\begin{itemize}
\item $X$ is a set.
\item $|\cdot|$ is a function $X\to\Z/2$ (the ``grading'').
\item $S$ is a function which assigns to each $x\in X$ a closed connected oriented $1$-manifold $S(x)$.
\item $\mc{O}$ assigns to each $x\in X$ a local system $\mc{O}_x$ over $S(x)$ which is locally isomorphic to $\Z$.
\item For $d\in\{0,1,2,3\}$ and $x_+,x_-\in X$ distinct, $M_d(x_+,x_-)$ is a smooth manifold of dimension $d$ (the ``moduli space'').
\item $e_+:M_d(x_+,x_-)\to S(x_+)$ and $e_-:M_d(x_+,x_-)\to S(x_-)$ are smooth maps (the ``evaluation maps'').
\item $M_d(x_+,x_-)$ is equipped with an orientation with values\footnote{If $M$ is a smooth manifold and $\mc{O}$ is a local system over $M$ which is locally isomorphic to $\Z$, then an ``orientation of $M$ with values in $\mc{O}$'' means a trivialization of $\mc{O}_M\tensor\mc{O}$, where $\mc{O}_M$ denotes the orientation sheaf of $M$.} in $e_+^*\mc{O}_{x_+}\tensor e_-^*\mc{O}_{x_-}$.
\end{itemize}
We require these moduli spaces and evaluation maps to satisfy the Grading, Fiber Product Transversality, Finiteness, and Compactification axioms below.
\end{definition}

\begin{description}
\item{(Grading)}
If $M_d(x_+,x_-)$ is nonempty, then
\begin{equation}
\label{eqn:grading}
d\equiv |x_+| - |x_-| \mod 2. 
\end{equation}
\item{(Fiber Product Transversality)}
If $x_1,x_2,x_3\in X$ are distinct and $d_1,d_2$ are nonnegative integers with $d_1+d_2\le 3$, then the fiber product
\[
M_{d_1}(x_1,x_2) \times_{S(x_2)}M_{d_2}(x_2,x_3)
\]
is cut out transversely.
\item{(Finiteness)}
For each $x_0\in X$, there are only finitely many tuples $(k,x_1,\ldots,x_k)$ where $k$ is a positive integer and $x_1,\ldots,x_k\in X$, such that there exist $d_1,\ldots,d_k\in\{0,1,2,3\}$ with $M_{d_i}(x_{i-1},x_i)\neq \emptyset$ for all $i=1,\ldots,k$.
\end{description}

To state the Compactification axiom, given $p_\pm\in S(x_\pm)$, define the following three subsets of ${M}_d(x_+,x_-)$:
\begin{equation}
\label{eqn:Mdconstrained}
\begin{split}
{M}_d(x_+,p_+,x_-) &= e_+^{-1}(p_+),\\
{M}_d(x_+,x_-,p_-) &= e_-^{-1}(p_-),\\
{M}_d(x_+,p_+,x_-,p_-) &= e_+^{-1}(p_+)\cap e_-^{-1}(p_-).
\end{split}
\end{equation}

\begin{convention}
\label{convention:constraints}
If $p_+$ is a regular value of $e_+$, then we orient $M_d(x_+,p_+,x_-)$ as a level set of $e_+$. If $p_-$ is a regular value of $e_-$, then we orient $M_d(x_+,x_-,p_-)$ as minus a level set of $e_-$. If $(p_+,p_-)$ is a regular value of $e_+\times e_-$, then we orient $M_d(x_+,p_+,x_-,p_-)$ as a level set of $e_+$ on $M_d(x_+,x_-,p_-)$.
\end{convention}

\begin{description}
\item{(Compactification)}
Let $x_+,x_-\in X$ be distinct, and let $(p_+,p_-)\in S(x_+)\times S(x_-)$ be generic. ``Genericity'' includes but is not limited to the following:
\begin{itemize}
\item $p_+$ is a regular value of all evaluation maps $e_+:M_d(x_+,x_0)\to S(x_+)$ for $d\le 2$.
\item $p_-$ is a regular value of all evaluation maps $e_-:M_d(x_0,x_-)\to S(x_-)$ for $d\le 2$.
\item $(p_+,p_-)$ is a regular value of
\[
e_+\times e_-:M_d(x_+,x_-)\longrightarrow S(x_+)\times S(x_-)
\]
for $d\le 3$.
\item All fiber products on the right hand sides of \eqref{eqn:compactification2-}, \eqref{eqn:compactification2+} and \eqref{eqn:compactification3} below are cut out transversely.
\end{itemize}
Then:
\begin{description}
\item{(a)} The moduli space $M_0(x_+,x_-)$ is compact, i.e.\ finite.
%, $M_1(x_+,p_+,x_-)$, $M_1(x_+,x_-,p_-)$, and $M_2(x_+,p_+,x_-,p_-)$ are compact, i.e.\ finite.
\item{(b)} The moduli space $M_1(x_+,x_-)$ has a compactification $\overline{M}_1(x_+,x_-)$, whose boundary has an identification\footnote{More precisely, we should say that part of the data of the Morse-Bott system is the compactification $\overline{M}_1(x_+,x_-)$ and the identification \eqref{eqn:compactification1}. A similar remark applies to the rest of the compactification axiom here and other compactification axioms later.}
\begin{equation}
\label{eqn:compactification1}
\partial\overline{M}_1(x_+,x_-) = \coprod_{\substack{x_0\neq x_+,x_- \\ d_++d_-=1}}
(-1)^{d_+}
M_{d_+}(x_+,x_0)\times_{S(x_0)}M_{d_-}(x_0,x_-).
\end{equation}
\item{(c)}
The moduli space $M_2(x_+,x_-,p_-)$ has a compactification $\overline{M}_2(x_+,x_-,p_-)$, whose boundary has an identification
\begin{equation}
\label{eqn:compactification2-}
\partial\overline{M}_2(x_+,x_-,p_-) = \coprod_{\substack{x_0\neq x_+,x_- \\ d_++d_-=2}}
(-1)^{d_+}
M_{d_+}(x_+,x_0)\times_{S(x_0)}M_{d_-}(x_0,x_-,p_-).
\end{equation}
The moduli space $M_2(x_+,p_+,x_-)$ has a compactification $\overline{M}_2(x_+,p_+,x_-)$, whose boundary has an identification
\begin{equation}
\label{eqn:compactification2+}
\partial\overline{M}_2(x_+,p_+,x_-) = \coprod_{\substack{x_0\neq x_+,x_- \\ d_++d_-=2}}
(-1)^{d_+-1}
M_{d_+}(x_+,p_+,x_0)\times_{S(x_0)}M_{d_-}(x_0,x_-).
\end{equation}
\item{(d)}
The moduli space
$M_3(x_+,p_+,x_-,p_-)$ has a compactification
$\overline{M}_3(x_+,p_+,x_-,p_-)$, whose boundary has an identification
\begin{equation}
\label{eqn:compactification3}
\partial\overline{M}_3(x_+,p_+,x_-,p_-) = \coprod_{\substack{x_0\neq x_+,x_- \\ d_++d_-=3}}
(-1)^{d_+-1}
M_{d_+}(x_+,p_+,x_0)\times_{S(x_0)}M_{d_-}(x_0,x_-,p_-).
\end{equation}
\end{description}
In each of the identifications \eqref{eqn:compactification1}, \eqref{eqn:compactification2-}, \eqref{eqn:compactification2+}, and \eqref{eqn:compactification3}, the boundary orientation on the left hand side agrees with the fiber product orientation on the right hand side.
In (b), (c), and (d), the evaluation maps $e_\pm$ on $M_1(x_+,x_-)$ etc.\ extend continuously to the compactifications, and on the boundaries satisfy
\[
e_\pm(u_+,u_-)=e_\pm(u_\pm).
\]
\begin{description}
\item{(e)} 
The right hand sides of \eqref{eqn:compactification1}, \eqref{eqn:compactification2-}, \eqref{eqn:compactification2+} and \eqref{eqn:compactification3} would not include any extra points if we used compactifications. For example, for \eqref{eqn:compactification1}, this means that if $x_+, x_0, x_-$ are distinct, then
\[
\begin{split}
M_0(x_+,x_0) \times_{S(x_0)} \left(\overline{M}_1(x_0,x_-)\setminus M_1(x_0,x_-)\right) &= \emptyset,\\
\left(\overline{M}_1(x_+,x_0)\setminus M_1(x_+,x_0)\right) \times_{S(x_0)} M_0(x_0,x_-) &= \emptyset.
\end{split}
\]
For \eqref{eqn:compactification2-}, this means that if $x_+, x_0, x_-$ are distinct, and if $p_-\in S(x_-)$ is generic, then
\[
\begin{split}
M_0(x_+,x_0) \times_{S(x_0)} \left(\overline{M}_2(x_0,x_-,p_-)\setminus M_2(x_0,x_-,p_-)\right) &= \emptyset,\\
\left(\overline{M}_1(x_+,x_0)\setminus M_1(x_+,x_0)\right) \times_{S(x_0)} M_1(x_0,x_-,p_-) &= \emptyset.
\end{split}
\]
\end{description}
\end{description}

\begin{remark}
The Grading axiom is needed only to obtain a $\Z/2$ grading on the cascade homology of a Morse-Bott system. One can also modify this axiom to obtain a relative $\Z/N$-grading on the cascade homology; to do this, one requires that the grading difference of two elements of $X$ is a well-defined element of $\Z/N$, such that $M_d(x_+,x_-)$ is nonempty only if $|x_+|-|x_-|\equiv d\mod N$.
\end{remark}

\begin{remark}
In many cases of interest, the following stronger version of the Finiteness axiom holds: there is an ``action'' function $A:X\to\R$ such that (i) for each $L\in\R$, there are only finitely many $x\in X$ with $A(x)<L$, and (ii) if $M_d(x_+,x_-)\neq\emptyset$ then $A(x_+)>A(x_-)$.
\end{remark}

\begin{remark}
A stronger version of parts (b)--(d) of the Compactification axiom would be that $M_d(x_+,x_-)$ has a compactification to a smooth manifold with corners $\overline{M}_d(x_+,x_-)$ for $d=1,2,3$, whose codimension $1$ stratum is given by
\[
\partial \overline{M}_d(x_+,x_-) = \coprod_{\substack{x_0\neq x_+,x_- \\ d_++d_-=d}} (-1)^{d_+}M_{d_+}(x_+,x_0) \times_{S(x_0)} M_{d_-}(x_0,x_-).
\]
Equations \eqref{eqn:compactification1}, \eqref{eqn:compactification2-}, \eqref{eqn:compactification2+}, and \eqref{eqn:compactification3} would follow from this, and this is our motivation for the signs in those equations.
\end{remark}

\begin{remark}
Part (e) of the Compactification axiom holds automatically if we know two additional properties: (i) Each of the compactifications in \eqref{eqn:compactification1}, \eqref{eqn:compactification2-}, \eqref{eqn:compactification2+}, and \eqref{eqn:compactification3} does not include any additional points aside from the boundary points of the compactification. That is, as a set we have $\overline{M}_1(x_+,x_-)\setminus M_1(x_+,x_-) = \partial M_1(x_+,x_-)$ etc. (ii) Fiber Product Transversality also holds for triple fiber products when the sum of the dimensions of the factors is at most $3$.
\end{remark}

\subsection{Morphisms of Morse-Bott systems}

We now define a morphism of Morse-Bott systems. This is very similar to the definition of a Morse-Bott system, but some signs are changed in the compactification axiom: compare equations \eqref{eqn:grading}, \eqref{eqn:compactification1}, \eqref{eqn:compactification2-}, \eqref{eqn:compactification2+}, and \eqref{eqn:compactification3} with equations \eqref{eqn:gradingphi}, \eqref{eqn:compactiphi1}, \eqref{eqn:compactiphi2-}, \eqref{eqn:compactiphi2+}, and \eqref{eqn:compactiphi3} respectively. For more about these sign changes see \S\ref{sec:conjugate}.

\begin{definition}
\label{def:morphism}
Let $A_1=(X_1,|\cdot|_1,S_1,\mc{O}^1,M_*^1,e_\pm^1)$ and $A_2=(X_2,|\cdot|_2,S_2,\mc{O}^2,M_*^2,e_\pm^2)$ be Morse-Bott systems. A {\em morphism\/} $\Phi$ of Morse-Bott systems from $A_1$ to $A_2$ consists of the following data for each $x_1\in X_1$, $x_2\in X_2$, and $d\in\{0,1,2,3\}$:
\begin{itemize}
\item A ``moduli space'' $\Phi_d(x_1,x_2)$, which is a smooth manifold of dimension $d$.
\item ``Evaluation maps'', which are smooth maps
\[
\begin{split}
e_+: \Phi_d(x_1,x_2) &\longrightarrow S_1(x_1),\\
e_-: \Phi_d(x_1,x_2) & \longrightarrow S_2(x_2).
\end{split}
\]
\item An orientation of $\Phi_d(x_1,x_2)$ with values in $e_+^*\mc{O}_{x_1}^1\tensor e_-^*\mc{O}_{x_2}^2$.
\end{itemize}
These are required to satisfy the following Grading, Finiteness, Fiber Product Transversality, and Compactification properties:
\begin{description}
\item{(Grading)} 
If $\Phi_d(x_1,x_2)$ is nonempty, then
\begin{equation}
\label{eqn:gradingphi}
d \equiv |x_1|_1 - |x_2|_2 + 1 \mod 2.
\end{equation}
\item{(Finiteness)}
For each $x_1\in X_1$, there exist only finitely many $x_2\in X_2$ such that $\Phi_d(x_1,x_2)$ is nonempty for some $d\in\{0,1,2,3\}$.
\item{(Fiber Product Transversality)}
If $x_1,x_1'\in X_1$ are distinct and $x_2\in X_2$, then all fiber products
\[
M^1_{d_1}(x_1,x_1')\times_{S_1(x_1')} \Phi_d(x_1',x_2)
\]
with $d_1+d\le 3$ are cut out transversely. Likewise, if $x_1\in X_1$ and $x_2,x_2'\in X_2$ are distinct, then all fiber products
\[
\Phi_d(x_1,x_2') \times_{S_2(x_2')} M^2_{d_2}(x_2',x_2)
\]
with $d+d_2\le 3$ are cut out transversely.
\item{(Compactification)}
Let $x_1\in X_1$ and $x_2\in X_2$. Let $(p_1,p_2)\in S_1(x_1)\times S_2(x_2)$ be generic. In particular, assume that $p_1$ is a regular value of all evaluation maps $e_+$ on $M_d$ and $\Phi_d$ for $d\le 2$; $p_2$ is a regular value of all evaluation maps $e_-$ on $M_d$ and $\Phi_d$ for $d\le 2$; and $(p_1,p_2)$ is a regular value of $e_+\times e_-$ on $\Phi_d(x_1,x_2)$ for $d\le 3$. Define $\Phi_d(x_1,p_1,x_2)$, $\Phi_d(x_1,x_2,p_2)$, and $\Phi_d(x_1,p_1,x_2,p_2)$ as in \eqref{eqn:Mdconstrained}. Assume also that all fiber products on the right hand sides of \eqref{eqn:compactiphi2-}, \eqref{eqn:compactiphi2+} and \eqref{eqn:compactiphi3} below are cut out transversely. Then: 
\begin{description}
\item{(a)}
The moduli space $\Phi_0(x_1,x_2)$ is compact, i.e.\ finite.
\item{(b)} $\Phi_1(x_1,x_2)$ has a compactification $\overline{\Phi}_1(x_1,x_2)$, whose boundary has an identification
\begin{equation}
\label{eqn:compactiphi1}
\begin{split}
\partial \overline{\Phi}_1(x_1,x_2) =& \coprod_{\substack{x_1'\in X_1\setminus\{x_1\}\\d_1+d=1}}
M_{d_1}^1(x_1,x_1') \times_{S_1(x_1')} \Phi_d(x_1',x_2)\\
& \bigsqcup \coprod_{\substack{x_2'\in X_2\setminus\{x_2\}\\ d+d_2=1}} 
(-1)^{d}
\Phi_d(x_1,x_2') \times_{S_2(x_2')} M_{d_2}^2(x_2',x_2).
\end{split}
\end{equation}
\item{(c)}
$\Phi_2(x_1,x_2,p_2)$ has a compactification $\overline{\Phi}_2(x_1,x_2,p_2)$, whose boundary has an identification
\begin{equation}
\label{eqn:compactiphi2-}
\begin{split}
\partial \overline{\Phi}_2(x_1,x_2,p_2) =& \coprod_{\substack{x_1'\in X_1\setminus\{x_1\}\\d_1+d=2}}
M_{d_1}^1(x_1,x_1') \times_{S_1(x_1')} \Phi_d(x_1',x_2,p_2)\\
& \bigsqcup \coprod_{\substack{x_2'\in X_2\setminus\{x_2\}\\ d+d_2=2}} 
(-1)^{d}
\Phi_d(x_1,x_2') \times_{S_2(x_2')} M_{d_2}^2(x_2',x_2,p_2).
\end{split}
\end{equation}
Likewise,
$\Phi_2(x_1,p_1,x_2)$ has a compactification $\overline{\Phi}_2(x_1,p_1,x_2)$, whose boundary has an identification
\begin{equation}
\label{eqn:compactiphi2+}
\begin{split}
\partial \overline{\Phi}_2(x_1,p_1,x_2) =& \coprod_{\substack{x_1'\in X_1\setminus\{x_1\}\\d_1+d=2}}
-
M_{d_1}^1(x_1,p_1,x_1') \times_{S_1(x_1')} \Phi_d(x_1',x_2)\\
& \bigsqcup \coprod_{\substack{x_2'\in X_2\setminus\{x_2\}\\ d+d_2=2}} 
(-1)^{d-1}
\Phi_d(x_1,p_1,x_2') \times_{S_2(x_2')} M_{d_2}^2(x_2',x_2).
\end{split}
\end{equation}
\item{(d)} $\Phi_3(x_1,p_1,x_2,p_2)$ has a compactification $\overline{\Phi}_3(x_1,p_1,x_2,p_2)$, whose boundary has an identification
\begin{equation}
\label{eqn:compactiphi3}
\begin{split}
\partial \overline{\Phi}_3(x_1,p_1,x_2,p_2) =& \coprod_{\substack{x_1'\in X_1\setminus\{x_1\}\\d_1+d=3}}
-
M_{d_1}^1(x_1,p_1,x_1') \times_{S_1(x_1')} \Phi_d(x_1',x_2,p_2)\\
& \bigsqcup \coprod_{\substack{x_2'\in X_2\setminus\{x_2\}\\ d+d_2=3}} 
(-1)^{d-1}
\Phi_d(x_1,p_1,x_2') \times_{S_2(x_2')} M_{d_2}^2(x_2',x_2,p_2).
\end{split}
\end{equation}
\end{description}
In each of the identifications \eqref{eqn:compactiphi1}, \eqref{eqn:compactiphi2-}, \eqref{eqn:compactiphi2+}, and \eqref{eqn:compactiphi3}, the boundary orientation on the left hand side agrees with the fiber product orientation on the right hand side.
In (b), (c), and (d), the evaluation maps $e_\pm$ extend continuously to the compactifications and satisfy $e_\pm(u_+,u_-)=e_\pm(u_\pm)$.
\begin{description}
\item{(e)}
As in part (e) of the Compactification axiom in the definition of Morse-Bott system, the right hand sides of \eqref{eqn:compactiphi1}, \eqref{eqn:compactiphi2-}, \eqref{eqn:compactiphi2+} and \eqref{eqn:compactiphi3} would not include any extra points if we used compactifications.  For example, for \eqref{eqn:compactiphi1}, this means that if $x_1,x_1'\in X_1$ are distinct and $x_2\in X_2$, then
\begin{equation}
\label{eqn:compacte1}
\begin{split}
(\overline{M}^1_1(x_1,x_1')\setminus M^1_1(x_1,x_1'))\times_{S_1(x_1')} \Phi_0(x_1',x_2) &= \emptyset,\\
M^1_0(x_1,x_1')\times_{S_1(x_1')} (\overline{\Phi}_1(x_1',x_2)\setminus \Phi_1(x_1',x_2)) &= \emptyset,
\end{split}
\end{equation}
and if $x_1\in X_1$ and $x_2,x_2'\in X_2$ are distinct, then
\begin{equation}
\label{eqn:compacte2}
\begin{split}
(\overline{\Phi}_1(x_1,x_2')\setminus \Phi_1(x_1,x_2')) \times_{S_2(x_2')} M^2_0(x_2',x_2) &= \emptyset,\\
\Phi_0(x_1,x_2') \times_{S_2(x_2')} (\overline{M}^2_1(x_2',x_2)\setminus M^2_1(x_2',x_2)) &= \emptyset.
\end{split}
\end{equation}
\end{description}
\end{description}
\end{definition}

\begin{example}
\label{ex:identity}
If $A=(X,|\cdot|,S,\mc{O},M_*,e_\pm)$ is a Morse-Bott system, then the identity morphism from $A$ to itself is defined as follows.
\begin{itemize}
\item
For all $x_1,x_2\in X$, we have
\begin{equation}
\label{eqn:identityempty}
\Phi_0(x_1,x_2) = \emptyset.
\end{equation}
\item
If $x\in X$, then
\[
\Phi_d(x,x) = \left\{\begin{array}{cl} S(x), & d=1,\\ \emptyset, & d\neq 1.
\end{array}\right.
\]
The evaluation maps
\[
e_\pm: \Phi_1(x,x) \longrightarrow S(x)
\]
are both defined to be the identity map. The orientation on $\Phi_1(x,x)$ with values in
\[
e_+^*\mc{O}_x\tensor e_-^*\mc{O}_x = \mc{O}_x\tensor\mc{O}_x = \Z
\]
agrees with the orientation on $S(x)$.
\item
If $x_1,x_2\in X$ are distinct, then
\begin{equation}
\label{eqn:identityextra}
\Phi_d(x_1,x_2) = \R\times M_{d-1}(x_1,x_2)
\end{equation}
for each $d\in\{1,2,3\}$. The evaluation maps $e_\pm$ on $\Phi_d(x_1,x_2)$ are pulled back from the evaluation maps on $M_{d-1}(x_1,x_2)$. The orientation on $\Phi_d(x_1,x_2)$ is the product orientation.
\end{itemize}
\end{example}

\begin{lemma}
\label{lem:identity}
If $A$ is a Morse-Bott system, then the identity morphism $\Phi$ from $A$ to itself, defined in Example~\ref{ex:identity}, is a morphism of Morse-Bott systems.
\end{lemma}

\begin{proof}
We need to check that the identity morphism $\Phi$ satisfies the Grading, Finiteness, Fiber Product Transversality, and Compactification axioms.

The Grading and Finiteness properties for $\Phi$ follow from the corresponding properties for $A$.

The Fiber Product Transversality property for $\Phi$ follows from the corresponding property for $A$, together with the fact that fiber products with the identity map are always cut out transversely.

We now prove the Compactification property for $\Phi$. Part (a) follows immediately from \eqref{eqn:identityempty}.

To prove part (b) of Compactification, suppose first that $x_1$ and $x_2$ are equal, to $x\in X$. We need to check is that if $x\in X$, then $\Phi_1(x,x)$ has a compactification $\overline{\Phi}_1(x,x)$ whose boundary has an identification
\begin{equation}
\label{eqn:identitycompactification}
\begin{split}
\partial \overline{\Phi}_1(x,x)  
=& \coprod_{\substack{x_1'\in X\setminus\{x\}\\d_1+d=1}}
M_{d_1}(x,x_1') \times_{S(x_1')} \Phi_d(x_1',x)\\
& \bigsqcup \coprod_{\substack{x_2'\in X\setminus\{x\}\\ d+d_2=1}} 
(-1)^{d}
\Phi_d(x,x_2') \times_{S(x_2')} M_{d_2}(x_2',x)
\end{split}
\end{equation}
as oriented $0$-manifolds. Since $\Phi_1(x,x)$ is already compact, we can (and must) compactify it by defining
\[
\overline{\Phi}_1(x,x) = \Phi_1(x,x).
\]
This will then satisfy \eqref{eqn:identitycompactification}, because the right hand side of \eqref{eqn:identitycompactification} is empty by \eqref{eqn:identityempty}.

Suppose now that $x_1\neq x_2$. To prove part (b) of Compactification in this case, define
\begin{equation}
\label{eqn:iddistinct}
\overline{\Phi}_1(x_1,x_2) = \overline{\R}\times M_0(x_1,x_2),
\end{equation}
where $\overline{\R}$ denotes the compactification of $\R$ obtained by adding two points at $\pm\infty$. We need an identification
\[
\begin{split}
\partial \overline{\Phi}_1(x_1,x_2) =& \coprod_{x_1'\neq x_1}
M_{1}(x_1,x_1') \times_{S(x_1')} \Phi_0(x_1',x_2)\\
& \bigsqcup \coprod_{x_1'\neq x_1} M_0(x_1,x_1') \times_{S(x_1')} \Phi_1(x_1',x_2)\\
& \bigsqcup \coprod_{x_2'\neq x_2} 
-\Phi_1(x_1,x_2') \times_{S(x_2')} M_{0}(x_2',x_2)\\
& \bigsqcup \coprod_{x_2'\neq x_2} \Phi_0(x_1,x_2') \times_{S(x_2')} M_1(x_2',x_2).
\end{split}
\]
The left hand side of this equation, by definition, is $M_0(x_1,x_2) \sqcup -M_0(x_1,x_2)$. On the right hand side, the first and last lines are empty by \eqref{eqn:identityempty}. The second line gives $M_0(x_1,x_2)$ when $x_1'=x_2$ by \eqref{eqn:fpid}, and is empty when $x_1'\neq x_2$ by the Fiber Product Transversality property of $A$. Likewise, the third line gives $-M_0(x_1,x_2)$.

Parts (c) and (d) of Compactification are proved similarly, setting
\[
\overline{\Phi}_2(x_1,x_2,p_2) = \overline{\R}\times M_1(x_1,x_2,p_2)
\]
and so forth.

Part (e) of the Compactification axiom follows from the Fiber Product Transversality property for $A$.
\end{proof}

\subsection{Composition of morphisms}

In order to compose morphisms, we need to make the following transversality hypotheses.

\begin{definition}
\label{def:composable}
Let $A_i=(X_i,|\cdot|_i,S_i,\mc{O}^i,M_*^i,e_\pm^i)$ be Morse-Bott systems for $i=1,2,3$. Let $\Phi$ be a morphism from $A_1$ to $A_2$, and let $\Psi$ be a morphism from $A_2$ to $A_3$. We say that the morphisms $\Phi$ and $\Psi$ are {\em composable\/} if the following hold:
\begin{description}
\item{(a)}
All fiber products of the form
\[
\Phi_{d_1}(x_1,x_2)\times_{S_2(x_2)}\Psi_{d_2}(x_2,x_3)
\]
with $d_1+d_2\le 3$ are cut out transversely.
\item{(b)}
All fiber products of the form
\begin{gather*}
M^1_d(x_1',x_1)\times_{S_1(x_1)} \Phi_{d_1}(x_1,x_2)\times_{S_2(x_2)}\Psi_{d_2}(x_2,x_3),\\ \Phi_{d_1}(x_1,x_2)\times_{S_2(x_2)}\Psi_{d_2}(x_2,x_3)\times_{S_3(x_3)} M^3_d(x_3,x_3')
\end{gather*}
with $d+d_1+d_2\le 4$ are cut out transversely.
\item{(c)}
All fiber products of the form
\begin{gather*}
(\overline{\Phi}_1(x_1,x_2)\setminus\Phi_1(x_1,x_2)) \times_{S_2(x_2)}\Psi_d(x_2,x_3),\\
\Phi_d(x_1,x_2)\times_{S_2(x_2)} (\overline{\Psi}_1(x_2,x_3)\setminus\Psi_1(x_2,x_3)),\\
(\overline{\Phi}_1(x_1,x_2)\setminus\Phi_1(x_1,x_2)) \times_{S_2(x_2)} (\overline{\Psi}_1(x_2,x_3)\setminus\Psi_1(x_2,x_3))
\end{gather*}
with $d\le 1$ are cut out transversely. (In particular, the first two are empty when $d=0$, and the third is always empty.)
\item{(d)}
All of the following fiber products are empty:
\[
\begin{split}
(\overline{M}^1_1(x_1,x_1')\setminus M^1_1(x_1,x_1'))\times_{S_1(x_1')}\Phi_1(x_1',x_2)\times_{S_2(x_2)}\Psi_0(x_2,x_3) &= \emptyset,\\
(\overline{\Phi}_1(x_1,x_2^+)\setminus \Phi_1(x_1,x_2^+))\times_{S_2(x_2^+)} M^2_1(x_2^+,x_2^-) \times_{S_2(x_2^-)} \Psi_0(x_2^-,x_3) &= \emptyset,\\
\Phi_0(x_1,x_2^+)\times_{S_2(x_2^+)} M^2_1(x_2^+,x_2^-) \times_{S_2(x_2^-)} (\overline{\Psi}_1(x_2^-,x_3)\setminus \Psi_1(x_2^-,x_3)) &= \emptyset,\\
\Phi_0(x_1,x_2) \times_{S_2(x_2)} \Psi_1(x_2,x_3') \times_{S_3(x_3')} (\overline{M}^3_1(x_3',x_3)\setminus M^3_1(x_3',x_3)) &= \emptyset.
\end{split}
\]
\item{(e)}
Analogues of conditions (c) and (d) hold in which one adds a point constraint at a generic point $p_1\in S_1(x_1)$ and increases the dimension of the first factor by one, and/or adds add a point constraint at a generic point $p_3\in S_3(x_3)$ and increases the dimension of the last factor by one.
\end{description}
\end{definition}

\begin{definition}
\label{def:composition}
Under the assumptions of Definition~\ref{def:composable}, suppose that the morphisms $\Phi$ and $\Psi$ are composable. The {\em composition\/} of $\Phi$ and $\Psi$ is a morphism $\Psi\circ\Phi$ from $A_1$ to $A_3$ defined as follows: If $x_1\in X_1$ and $x_3\in X_3$ and $d\in\{0,1,2,3\}$, then
\begin{equation}
\label{eqn:composition}
(\Psi\circ\Phi)_d(x_1,x_3) = \coprod_{\substack{x_2\in X_2\\ d_1+d_2=d+1}} \Phi_{d_1}(x_1,x_2)\times_{S_2(x_2)} \Psi_{d_2}(x_2,x_3),
\end{equation}
with the fiber product orientation. This is well defined by part (a) of the definition of composability. The evaluation maps $e_+:(\Psi\circ\Phi)_d(x_1,x_3)\to S_1(x_1)$ and $e_-:(\Psi\circ\Phi)_d(x_1,x_3)\to S_3(x_3)$ are defined by $e_\pm(u_+,u_-)=e_\pm(u_\pm)$.
\end{definition}

\begin{proposition}
\label{prop:composition}
Under the assumptions of Definition~\ref{def:composable}, if the morphisms $\Phi$ and $\Psi$ are composable, then the composition $\Psi\circ\Phi$ is a morphism\footnote{More precisely, we can define compactifications $\overline{(\Psi\circ\Phi)}_1(x_1,x_3)$ etc.\ in a canonical way in order to make $\Psi\circ\Phi$ into a morphism of Morse-Bott systems.} of Morse-Bott systems.
\end{proposition}

\begin{proof}
The Grading and Finiteness properties for $\Psi\circ\Phi$ follow from the Grading and Finiteness properties for $\Phi$ and $\Psi$.

The Fiber Product Transversality property for $\Psi\circ\Phi$ follows from part (b) of the assumption that $\Phi$ and $\Psi$ are composable.

To prove the Compactification property for $\Psi\circ\Phi$, let $x_1\in X_1$ and $x_3\in X_3$. We need to prove parts (a)--(e) of the Compactification property for $x_1$ and $x_3$.

(a) We need to prove that $(\Psi\circ\Phi)_0(x_1,x_3)$ is finite. Suppose to get a contradiction that $(\Psi\circ\Phi)_0(x_1,x_3)$ contains an infinite sequence of distinct elements $\{(u_1^i,u_2^i)\}_{i=1,\ldots}$. By the definition of $(\Psi\circ\Phi)_0(x_1,x_3)$, for each $i$, there is an element $x_2^i\in X_2$, and a pair of integers $(d_1^i,d_2^i)$ equal to $(1,0)$ or $(0,1)$, such that
\[
(u_1^i,u_2^i)\in \Phi_{d_1^i}(x_1,x_2^i)\times_{S_2(x_2^i)} \Psi_{d_2^i}(x_2^i,x_3).
\]
By the Finiteness property for $\Phi$ applied to $x_1$, we can pass to a subsequence such that all of the $x_2^i$ are equal to a single element $x_2\in X_2$. We can also pass to a further subsequence so that the $d_1^i$ are all equal to a single integer $d_1\in\{0,1\}$. Without loss of generality, $d_1=1$. Thus for all $i$ we have
\[
(u_1^i,u_2^i)\in \Phi_1(x_1,x_2)\times_{S_2(x_2)} \Psi_0(x_2,x_3).
\]
Since $\Psi_0(x_2,x_3)$ is finite, which we know by the Compactification property for $\Psi$, we can pass to a subsequence so that all of the $u_2^i$ are equal to a single element $u_2\in \Psi_0(x_2,x_3)$.  
By passing to a further subsequence, we can assume that the sequence $\{u_1^i\}_{i=1,\ldots}$ converges to a point $u_1^\infty$ in the compactification $\overline{\Phi}_1(x_1,x_2)$, which is provided by the Compactification property for $\Phi$. By part (c) of the assumption that $\Phi$ and $\Psi$ are composable, we cannot have $u_1^\infty\in\overline{\Phi}_1(x_1,x_2)\setminus\Phi_1(x_1,x_2)$. 
So $u_1^\infty\in \Phi_1(x_1,x_2)$.

By part (a) of the assumption that $\Phi$ and $\Psi$ are composable, the fiber product $\Phi_1(x_1,x_2)\times_{S(x_2)} \Psi_0(x_2,x_3)$ is cut out transversely, so the point $(u_1^\infty,u_2)$ is isolated in this fiber product. This contradicts the fact that there is a sequence of distinct points $(u_1^i,u_2)$ in the fiber product converging to it.

(b) We need to construct the compactification $\overline{(\Psi\circ\Phi)}_1(x_1,x_3)$. By definition,
\begin{equation}
\label{eqn:m113}
(\Psi\circ\Phi)_1(x_1,x_3)=\coprod_{\substack{x_2\in X_2 \\ d_1+d_2=2}} \Phi_{d_1}(x_1,x_2)\times_{S_2(x_2)} \Psi_{d_2}(x_2,x_3).
\end{equation}
We first define a preliminary compactification $\widetilde{\Psi\circ\Phi}_1(x_1,x_3)$, which is not the compactification we want; the latter will be obtained from the former by identifying some boundary points. The preliminary compactification is
\begin{equation}
\label{eqn:precomp}
\widetilde{(\Psi\circ\Phi)}_1(x_1,x_3)=\coprod_{\substack{x_2\in X_2 \\ d_1+d_2=2}} \overline{\Phi_{d_1}(x_1,x_2)\times_{S_2(x_2)} \Psi_{d_2}(x_2,x_3)},
\end{equation}
where the right hand side is to be interpreted as follows.
 When $(d_1,d_2)=(1,1)$, we set
\begin{equation}
\label{eqn:precomp11}
\overline{\Phi_{1}(x_1,x_2)\times_{S_2(x_2)} \Psi_{1}(x_2,x_3)} = \overline{\Phi}_1(x_1,x_2) \times_{S_2(x_2)} \overline{\Psi}_1(x_2,x_3).
\end{equation}
It follows from parts (a) and (c) of the assumption that $\Phi$ and $\Psi$ are composable that this is a topological $1$-manifold with boundary. When $(d_1,d_2)=(0,2)$, we cannot make an analogous definition because $\overline{\Psi}_2(x_2,x_3)$ is not defined. Instead, by \eqref{eqn:fpls} we can write
\[
\Phi_0(x_1,x_2)\times_{S_2(x_2)} \Psi_2(x_2,x_3) = \coprod_{u_1\in \Phi_0(x_1,x_2)}\varepsilon(u_1) \Psi_2(x_2,e_-(u_1),x_3).
\]
Here $\varepsilon(u_1)\in\{\pm1\}$ denotes the orientation of the point $u_1$ in $\Phi_0(x_1,x_2)$. We then define
\begin{equation}
\label{eqn:precomp02}
\overline{\Phi_0(x_1,x_2)\times_{S_2(x_2)} \Psi_2(x_2,x_3)} = \coprod_{u_1\in \Phi_0(x_1,x_2)}\varepsilon(u_1) \overline{\Psi}_2(x_2,e_-(u_1),x_3).
\end{equation}
The case $(d_1,d_2)=(2,0)$ is handled analogously.

To see that the preliminary compactification \eqref{eqn:precomp} is compact, note that by the finiteness property for $\Phi$, only finitely many triples $(x_2,d_1,d_2)$ give nonempty contributions to the right hand side of \eqref{eqn:precomp}. When $(d_1,d_2)=(1,1)$, the contribution \eqref{eqn:precomp11} is by definition compact. When $(d_1,d_2)=(0,2)$, the contribution \eqref{eqn:precomp02} is compact because $\Phi_0(x_1,x_2)$ is finite and $\overline{\Psi}_2(x_2,e_-(u_1),x_3)$ is compact by the Compactification properties for $\Phi$ and $\Psi$. Likewise, each contribution with $(d_1,d_2)=(2,0)$ is compact.

To proceed from the preliminary compactification to the actual compactification, consider the following three oriented $0$-manifolds:
\begin{equation}
\label{eqn:E123}
\begin{split}
E_1 &= \coprod_{\substack{x_1'\in X_1\setminus\{x_1\}\\x_2\in X_2\\d_1+d_2+d_3=2}}
 M_{d_1}^1(x_1,x_1') \times_{S_1(x_1')} \Phi_{d_2} (x_1',x_2) \times_{S_2(x_2)} \Psi_{d_3}(x_2,x_3),\\
E_2 &= \coprod_{\substack{x_2^+\neq x_2^-\in X_2\\d_1+d_2+d_3=2}} \Phi_{d_1}(x_1,x_2^+)\times_{S_2(x_2^+)} M_{d_2}^2(x_2^+,x_2^-)\times_{S_2(x_2^-)} \Psi_{d_3}(x_2^-,x_3),\\
E_3 &= \coprod_{\substack{x_2\in X_2 \\ x_3'\in X_3\setminus\{x_3\} \\ d_1+d_2+d_3=2}}
(-1)^{d_3-1}
\Phi_{d_1}(x_1,x_2) \times_{S_2(x_2)} \Psi_{d_2}(x_2,x_3') \times_{S_3(x_3')} M_{d_3}^3(x_3',x_3).
\end{split}
\end{equation}

We claim now that there is a map
\begin{equation}
\label{eqn:endsmap}
\phi: \partial\left(\widetilde{(\Psi\circ\Phi)}_1(x_1,x_3)\right) \longrightarrow E_1\sqcup E_2\sqcup E_3
\end{equation}
with the following properties:
\begin{description}
\item{(i)}
Each point in $E_1\sqcup E_3$ has exactly one inverse image under $\phi$. Each point in $E_1\sqcup E_3$ has the same orientation as its inverse image under $\phi$.
\item{(ii)}
Each point in $E_2$ has exactly two inverse images under $\phi$, 
and these two inverse images have opposite orientations.
\end{description}

Assuming (i) and (ii), we define
\[
\overline{(\Psi\circ\Phi)}_1(x_1,x_3) = \widetilde{(\Psi\circ\Phi)}_1(x_1,x_3)/\sim,
\]
where the equivalence relation $\sim$ identifies two points if they are on the boundary and $\phi$ maps them to the same point in $E_2$. 
By (i) and (ii), $\overline{(\Psi\circ\Phi)}_1(x_1,x_3)$ is an oriented topological 1-manifold with oriented boundary given by
\[
\partial\left(\overline{(\Psi\circ\Phi)}_1(x_1,x_3)\right) = E_1\sqcup E_3.
\]
This is the correct boundary, since we can rewrite
\[
\begin{split}
E_1 &= \coprod_{\substack{x_1'\in X_1\setminus\{x_1\}\\ d_1+d=1}}
M_{d_1}^1(x_1,x_1')\times_{S_1(x_1')} (\Psi\circ\Phi)_d(x_1',x_3),
\\
E_3 &= \coprod_{\substack{x_3'\in X_3\setminus\{x_3\}\\ d+d_3=1}}
(-1)^{d}
(\Psi\circ\Phi)_d(x_1,x_3') \times_{S_3(x_3')} M_{d_3}^3(x_3',x_3).
\end{split}
\]

To define the map \eqref{eqn:endsmap} and prove (i) and (ii), we now catalog all of the boundary points of $\widetilde{(\Psi\circ\Phi)}_1(x_1,x_3)$.  To shorten the equations, given $x_2\in X_2$ and $d_1,d_2\in\N$ with $d_1+d_2=2$, define
\[
N_{d_1,d_2}(x_2) = \overline{\Phi_{d_1}(x_1,x_2) \times_{S_2(x_2)} \Psi_{d_2}(x_2,x_3)}.
\]
The preliminary compactification is the disjoint union of the compact oriented $1$-manifolds $N_{d_1,d_2}(x_2)$. 

By equation \eqref{eqn:bfp} and part (b) of the Compactification property for $\Phi$ and $\Psi$, we have
\begin{equation}
\label{eqn:ends11}
\begin{split}
\partial N_{1,1}(x_2) =& \coprod_{\substack{x_1'\in X_1\setminus\{x_1\}\\ d_1+d=1}}
M_{d_1}^1(x_1,x_1')\times_{S_1(x_1')} \Phi_{d}(x_1',x_2) \times_{S_2(x_2)} \Psi_1(x_2,x_3) \\
&\bigsqcup 
\coprod_{\substack{x_2'\in X_2\setminus\{x_2\}\\d+d_2=1}}
(-1)^{d_2-1}
\Phi_d(x_1,x_2') \times_{S_2(x_2')} M_{d_2}^2(x_2',x_2) \times_{S_2(x_2)} \Psi_1(x_2,x_3)\\
&\bigsqcup
\coprod_{\substack{x_2'\in X_2\setminus\{x_2\}\\ d_2+d=1}}
\Phi_{1}(x_1,x_2) \times_{S_2(x_2)} M_{d_2}^2(x_2,x_2') \times_{S_2(x_2')} \Psi_{d}(x_2',x_3)\\
&\bigsqcup
\coprod_{\substack{x_3'\in X_3\setminus\{x_3\}\\ d+d_3=1}}
(-1)^{d_3-1}
\Phi_1(x_1,x_2) \times_{S_2(x_2)} \Psi_d(x_2,x_3') \times_{S_3(x_3')} M_{d_3}^3(x_3',x_3).
\end{split}
\end{equation}
Note that a priori, we should use $\overline{\Psi}_1(x_2,x_3)$ instead of $\Psi_1(x_2,x_3)$ in the first two terms on the right hand side, and $\overline{\Phi}_1(x_1,x_2)$ instead of $\Phi_1(x_1,x_2)$ and $\Psi_1(x_2,x_3)$ in the last two terms. However no points in $\overline{\Phi}_1(x_1,x_2)\setminus\Phi_1(x_1,x_2)$ or $\overline{\Psi}_1(x_2,x_3)\setminus\Psi_1(x_2,x_3)$ contribute to the corresponding fiber products, by condition (d) in the assumption that $\Phi$ and $\Psi$ are composable.

By equation \eqref{eqn:precomp02} and part (c) of the Compactification property for $\Psi$, we have
\begin{equation}
\label{eqn:ends02}
\begin{split}
\partial N_{0,2}(x_2) =& 
\coprod_{\substack{x_2'\in X_2\setminus\{x_2\}\\ d_2+d=2}}
-
\Phi_0(x_1,x_2) \times_{S_2(x_2)} M^2_{d_2}(x_2,x_2') \times_{S_2(x_2')} \Psi_{d}(x_2',x_3) \\
& \bigsqcup \coprod_{\substack{x_3'\in X_3\setminus\{x_3\}\\ d+d_3=2}} 
(-1)^{d_3-1}
\Phi_0(x_1,x_2) \times_{S_2(x_2)} \Psi_{d}(x_2,x_3') \times_{S_3(x_3')} M^3_{d_3}(x_3',x_3).
\end{split}
\end{equation}
Similarly, by part (c) of the Compactification property for $\Phi$, we have
\begin{equation}
\label{eqn:ends20}
\begin{split}
\partial N_{2,0}(x_2) =&
\coprod_{\substack{x_1'\in X_1\setminus\{x_1\}\\ d_1+d=2}}
M^1_{d_1}(x_1,x_1') \times_{S_1(x_1')} \Phi_{d}(x_1',x_2) \times_{S_2(x_2)} \Psi_0(x_2,x_3) \\
& \bigsqcup \coprod_{\substack{x_2'\in X_2\setminus\{x_2\}\\ d+d_2=2}} 
(-1)^{d_2}
\Phi_{d}(x_1,x_2') \times_{S_2(x_2')} M^2_{d_2}(x_2',x_2) \times_{S_2(x_2)} \Psi_0(x_2,x_3).
\end{split}
\end{equation}

Now all of the boundary points of $\widetilde{(\Psi\circ\Phi)}_1(x_1,x_3)$ are listed on the right hand sides of \eqref{eqn:ends11}, \eqref{eqn:ends02}, and \eqref{eqn:ends20}. Each element of the right hand side of one of these three equations in which the symbol $x_k'$ appears corresponds to a point in $E_k$. This defines the map \eqref{eqn:endsmap}.

To prove (i) and (ii), we need to count how many times each point in \eqref{eqn:E123} appears on the right hand side of \eqref{eqn:ends11}, \eqref{eqn:ends02}, or \eqref{eqn:ends20}, as $x_2$ ranges over $X_2$, and compare orientations. Note that the fiber products in \eqref{eqn:E123} are empty when $(d_1,d_2,d_3)$ equals $(2,0,0)$ or $(0,0,2)$. The remaining possibilities for $(d_1,d_2,d_3)$ are $(1,1,0)$, $(1,0,1)$, $(0,1,1)$, and $(0,2,0)$. We then see by inspection that each point in $E_1$ or $E_3$ appears exactly once on the right hand side of \eqref{eqn:ends11}, \eqref{eqn:ends02}, or \eqref{eqn:ends02}, with the same sign as in \eqref{eqn:E123}. On the other hand, each point in $E_2$ appears exactly twice on the right hand side of \eqref{eqn:ends02}, \eqref{eqn:ends02}, or \eqref{eqn:ends20}, once with $x_2=x_2^+$ and once with $x_2=x_2^-$, and these two appearances have opposite signs.

Parts (c) and (d) of the Compactification property are proved by the same argument as part (b), but with point constraints at $p_+$ and/or $p_-$ inserted everywhere.

To prove part (e) of the Compactification property, we will just explain \eqref{eqn:compacte1}, as \eqref{eqn:compacte2} is proved symmetrically, and the rest is proved analogously with point constraints at $p_+$ and/or $p_-$ inserted.

To prove the first line of \eqref{eqn:compacte1}, we need to show that
\[
(\overline{M}^1_1(x_1,x_1') \setminus M^1_1(x_1,x_1')) \times_{S_1(x_1')} \Phi_{d_1}(x_1',x_2)\times\Psi_{d_2}(x_2,x_3) = \emptyset
\]
whenever $d_1+d_2=1$. When $d_1=0$ this follows from the fact that $\Phi$ is a morphism. When $d_1=1$ this follows from condition (d) in the definition of composable.

To prove the second line of \eqref{eqn:compacte1}, we need to show that
\begin{equation}
\label{eqn:compacte12}
M^1_0(x_1,x_1') \times_{S_1(x_1')} (\overline{(\Psi\circ\Phi)}_1(x_1',x_3) \setminus (\Psi\circ\Phi)_1(x_1',x_3)) = \emptyset.
\end{equation}
The second factor, $\overline{(\Psi\circ\Phi)}_1\setminus(\Psi\circ\Phi)_1$, consists of points in $(\overline{\Phi}_1\setminus\Phi_1)\times_{S(x_2)}\Psi_1$ and $\Phi_1\times_{S(x_2)}(\overline{\Psi}_1\setminus\Psi_1)$, as well as points as in \eqref{eqn:E123}. In each case, the contributions to the fiber product \eqref{eqn:compacte12} are empty, either by condition (d) in the definition of composable, or by the fact that $\Phi$ is a morphism.
\end{proof}

\begin{remark}
\label{rem:identity}
Our definition of ``identity morphism'' is a slight abuse of terminology, for the following reason. Let $A_1$ and $A_2$ be Morse-Bott systems, let $\Phi$ be a morphism from $A_1$ to $A_2$, and let $I^i$ denote the identity morphism from $A_i$ to itself for $i=1,2$. Then:
\begin{itemize}
\item
$I^1$ and $\Phi$ are not necessarily composable; likewise $\Phi$ and $I^2$ are not necessarily composable. Composability with the identity requires $\Phi$ to satisfy slightly stronger transversality conditions than in the definition of ``morphism''.
\item
Even when composability holds, the compositions $\Phi\circ I^1$ and $I^2\circ\Phi$ are not quite equal to $\Phi$; the moduli spaces for the compositions are larger than the moduli spaces for $\Phi$ because of additional contributions coming from \eqref{eqn:identityextra}. For example, for $x_1,x_2$ distinct we have
\[
\begin{split}
(I^2\circ\Phi)_1(x_1,x_2) =& \Phi_1(x_1,x_2)\\
& \bigsqcup \R\times\coprod_{x_2'\in X_2\setminus\{x_2\}} \Phi_1(x_1,x_2')\times_{S(x_2')} M^2_0(x_2',x_2).
\end{split}
\]
The compactification $\overline{(I^2\circ\Phi)}_1(x_1,x_2)$ then includes an extra piece
\begin{equation}
\label{eqn:cep}
\overline{\R} \times\coprod_{x_2'\in X_2\setminus\{x_2\}} \Phi_1(x_1,x_2')\times_{S(x_2')} M^2_0(x_2',x_2).
\end{equation}
The evaluation maps are constant on each component of \eqref{eqn:cep}. In the construction in Proposition~\ref{prop:composition}, each component of \eqref{eqn:cep} is glued onto the corresponding boundary point of $\overline{\Phi}_1(x_1,x_2)$ at the point where the $\overline{\R}$ coordinate is $-\infty$.
\end{itemize}
\end{remark}

\begin{remark}
Although we will not need this, one can also show that the composition of three morphisms is associative, assuming that all morphisms and pairwise compositions in question are composable. This follows from \eqref{eqn:composition} and the associativity of fiber product, together with a check that the compactifications agree.
\end{remark}

\subsection{Homotopies of morphisms}

\begin{definition}
\label{def:homotopy}
Let $A_1=(X_1,\|\cdot\|_1,S_1,\mc{O}^1,M_*^1,e_\pm^1)$ and $A_2=(X_2,\|\cdot\|_2,S_2,\mc{O}^2,M_*^2,e_\pm^2)$ be Morse-Bott systems. Let $\Phi$ and $\Phi'$ be morphisms from $A_1$ to $A_2$. A {\em homotopy\/} $K$ from $\Phi$ to $\Phi'$ consists of the following data for each $x_1\in X_1$, $x_2\in X_2$, and $d\in\{0,1,2,3\}$:
\begin{itemize}
\item
A ``moduli space'' $K_d(x_1,x_2)$, which is a smooth manifold of dimension $d$.
\item
``Evaluation maps'', which are smooth maps
\[
\begin{split}
e_+: K_d(x_1,x_2) &\longrightarrow S_1(x_1),\\
e_-: K_d(x_1,x_2) &\longrightarrow S_2(x_2).
\end{split}
\]
\item
An orientation of $K_d(x_1,x_2)$ with values in $e_+^*\mc{O}^1_{x_1}\tensor e_-^*\mc{O}^2_{x_2}$. 
\end{itemize}
These are required to satisfy the following Grading, Finiteness, Fiber Product Transversality, and Compactification properties:
\begin{description}
\item{(Grading)}
If $K_d(x_1,x_2)$ is nonempty, then
\[
d \equiv |x_1|_1 - |x_2|_2 + 2 \mod 2.
\]
\item{(Finiteness)}
For each $x_1\in X_1$, there exist only finitely many $x_2\in X^2$ such that $K_d(x_1,x_2)$ is nonempty for some $d\in\{0,1,2,3\}$.
\item{(Fiber Product Transversality)}
This condition is the same as in the definition of ``morphism'', but with $\Phi_d$ replaced by $K_d$.
\item{(Compactification)}
Let $x_1\in X^1$ and $x_2\in X^2$. Let $(p_1,p_2)\in S_1(x_1)\times S_2(x_2)$ be generic. Define $K_d(x_1,p_1,x_2)$, $K_d(x_1,x_2,p_2)$, and $K_d(x_1,p_1,x_2,p_2)$ as in \eqref{eqn:Mdconstrained}. Suppose that $p_1$ is a regular value of all evaluation maps $e_+$ on $M^1_d$, $\Phi_d$, $\Phi'_d$, and $K_d$ for $d\le 2$; $p_2$ is a regular value of all evaluation maps $e_-$ on $M^2_d$, $\Phi_d$, $\Phi'_d$, and $K_d$ for $d\le 2$; and $(p_1,p_2)$ is a regular value of $e_+\times e_-$ on $\Phi_d(x_1,x_2)$, $\Phi'_d(x_1,x_2)$, and $K_d(x_1,x_2)$ for $d\le 3$.
Then:
\begin{description}
\item{(a)}
The moduli space $K_0(x_1,x_2)$ is compact, i.e.\ finite.
\item{(b)} $K_1(x_1,x_2)$ has a compactification $\overline{K}_1(x_1,x_2)$, whose boundary has an identification
\begin{equation}
\label{eqn:compactopy1}
\begin{split}
\partial \overline{K}_1(x_1,x_2) =& -\Phi_0(x_1,x_2) \bigsqcup \Phi'_0(x_1,x_2)\\
& \bigsqcup \coprod_{\substack{x_1'\in X^1\setminus\{x_1\}\\d_1+d=1}}
(-1)^{d_1} M_{d_1}^1(x_1,x_1') \times_{S_1(x_1')} K_d(x_1',x_2)\\
& \bigsqcup \coprod_{\substack{x_2'\in X^2\setminus\{x_2\}\\ d+d_2=1}} (-1)^d K_d(x_1,x_2') \times_{S_2(x_2')} M_{d_2}^2(x_2',x_2).
\end{split}
\end{equation}
\item{(c)}
$K_2(x_1,x_2,p_2)$ has a compactification $\overline{K}_2(x_1,x_2,p_2)$, whose boundary has an identification
\begin{equation}
\label{eqn:compactopy2-}
\begin{split}
\partial \overline{K}_2(x_1,x_2,p_2) =& -\Phi_1(x_1,x_2,p_2) \bigsqcup \Phi'_1(x_1,x_2,p_2)\\
& \bigsqcup \coprod_{\substack{x_1'\in X^1\setminus\{x_1\}\\d_1+d=2}} (-1)^{d_1} M_{d_1}^1(x_1,x_1') \times_{S_1(x_1')} K_d(x_1',x_2,p_2)\\
& \bigsqcup \coprod_{\substack{x_2'\in X^2\setminus\{x_2\}\\ d+d_2=2}} (-1)^d K_d(x_1,x_2') \times_{S_2(x_2')} M_{d_2}^2(x_2',x_2,p_2).
\end{split}
\end{equation}
Similarly, $K_2(x_1,p_1,x_2)$ has a compactification $\overline{K}_2(x_1,p_1,x_2)$, whose boundary has an identification
\begin{equation}
\label{eqn:compactopy2+}
\begin{split}
\partial \overline{K}_2(x_1,p_1,x_2) =& \Phi_1(x_1,p_1,x_2) \bigsqcup -\Phi'_1(x_1,p_1,x_2)\\
& \bigsqcup \coprod_{\substack{x_1'\in X^1\setminus\{x_1\}\\d_1+d=2}} (-1)^{d_1-1} M_{d_1}^1(x_1,p_1,x_1') \times_{S_1(x_1')} K_d(x_1',x_2)\\
& \bigsqcup \coprod_{\substack{x_2'\in X^2\setminus\{x_2\}\\ d+d_2=2}} (-1)^{d-1} K_d(x_1,p_1,x_2') \times_{S_2(x_2')} M_{d_2}^2(x_2',x_2).
\end{split}
\end{equation}
\item{(d)} $K_3(x_1,p_1,x_2,p_2)$ has a compactification $\overline{K}_3(x_1,p_1,x_2,p_2)$, whose boundary has an identification
\begin{equation}
\label{eqn:compactopy3}
\begin{split}
\partial \overline{K}_3(x_1,p_1,x_2,p_2) =& \Phi_2(x_1,p_1,x_2,p_2) \bigsqcup -\Phi_2'(x_1,p_1,x_2,p_2)\\
& \bigsqcup \coprod_{\substack{x_1'\in X^1\setminus\{x_1\}\\d_1+d=3}} (-1)^{d_1-1} M_{d_1}^1(x_1,p_1,x_1') \times_{S_1(x_1')} K_d(x_1',x_2,p_2)\\
& \bigsqcup \coprod_{\substack{x_2'\in X^2\setminus\{x_2\}\\ d+d_2=3}} (-1)^{d-1} K_d(x_1,p_1,x_2') \times_{S_2(x_2')} M_{d_2}^2(x_2',x_2,p_2).
\end{split}
\end{equation}
\end{description}
In (b), (c), and (d), the evaluation maps $e_\pm$ extend continuously to the compactifications and satisfy $e_\pm(u_+,u_-)=e_\pm(u_\pm)$.
\begin{description}
\item{(e)} As in part (e) of the Compactification axiom in the definition of Morse-Bott system, the right hand sides of \eqref{eqn:compactopy1}, \eqref{eqn:compactopy2-}, \eqref{eqn:compactopy2+} and \eqref{eqn:compactopy3} would not include any extra points if we used compactifications.
\end{description}
\end{description}
\end{definition}

\section{Cascade homology}
\label{sec:cascade}

In this section, we define the cascade homology of a Morse-Bott system. We show that cascade homology is functorial with respect to morphisms of Morse-Bott systems, and that the induced maps on cascade homology are invariant under homotopy of morphisms.

\subsection{Setup}
\label{sec:cascadesetup}

Let $A=(X,|\cdot|,S,\mc{O},M_*,e_\pm)$ be a Morse-Bott system. 
We are going to define its cascade homology, which is a $\Z/2$-graded $\Z$-module, denoted by $H_*^\ca(A)$. To do so, we need to make a generic choice of a point $p_x\in S_x$ for each $x\in X$. In particular, we require that:
\begin{itemize}
\item
$p_x$ is a regular value of all evaluation maps
\[
\begin{split}
e_-:M_d(x_+,x)&\longrightarrow S_x,\\
e_+:M_d(x,x_-)&\longrightarrow S_x
\end{split}
\]
for each $x\in X$ and $d\le 2$.
\item
 $(p_{x_+},p_{x_-})$ is a regular value of
\[
e_+\times e_-: M_d(x_+,x_-) \longrightarrow S(x_+)\times S(x_-)
\]
for each pair of distinct points $x_+,x_-\in X$ and each $d\le 3$.
\end{itemize}
We denote the set of choices $\{p_x\}_{x\in X}$ by $\mc{P}$. Below we will define the cascade chain complex $(C_*^\ca(A,\mc{P}),\partial)$. The homology of this chain complex will be denoted by $H_*^\ca(A,\mc{P})$. We will later show that this homology does not depend on $\mc{P}$, and so we can denote it by $H_*^\ca(A)$.

The chain complex $C_*^\ca(A,\mc{P})$ is the sum over $x\in X$ of two copies of $\mc{O}_x(p_x)$. To describe this a bit more conveniently, fix a generator of $\mc{O}_x(p_x)$ for each $x\in S$. Then the chain complex $C_*^\ca(A,\mc{P})$ is freely generated over $\Z$, with two generators for each $x\in X$. We denote these generators by $\widehat{x}$ and $\widecheck{x}$. The mod 2 grading of $\widecheck{x}$ equals $|x|$, while the mod 2 grading of $\widehat{x}$ equals $|x|+1$.

\subsection{Cascade moduli spaces: definition}
\label{sec:cmsd}

To define the differential $\partial$ on the chain complex $C_*^\ca(A,\mc{P})$, we need to introduce cascade moduli spaces. Roughly speaking, we will consider ``cascades'' that start at $\widehat{x}_+$ or $\widecheck{x}_+$ and end at $\widehat{x}_-$ or $\widecheck{x}_-$. When we start at $\widecheck{x}_+$ there is an initial point constraint, and we end at $\widehat{x}_-$ there is a final point constraint. Now for the precise definitions.

\subsubsection{Notation and simple cases}

To define cascade moduli spaces, we need the following notation.
Define the ``cyclic fiber product''
\begin{equation}
\label{eqn:cfp}
M_{d_1}(x_0,x_1)
\underset{S(x_1)}{\circlearrowleft} M_{d_2}(x_1,x_2) \underset{S(x_2)}{\circlearrowleft} \cdots \underset{S(x_{k-1})}{\circlearrowleft} M_{d_{k-1}}(x_{k-1},x_k)
\end{equation}
to be the set of $k$-tuples $(u_1,\ldots,u_k)$ such that:
\begin{itemize}
\item
$u_i\in M_{d_i}(x_{i-1},x_i)$ for each $i=1,\ldots,k$.
\item
For each $i=1,\ldots,k-1$, the points $p_{x_i},e_-(u_i),e_+(u_{i+1})$ are distinct and positively cyclically ordered with respect to the orientation of $S(x_i)$.
\end{itemize}
In \eqref{eqn:cfp}, we can also replace the first factor by $M_{d_1}(x_0,p_{x_0},x_1)$, in which case we require that $u_1\in M_{d_1}(x_0,p_0,x_1)$; and we can replace the last factor by $M_{d_k}(x_{k-1},x_k,p_{x_k})$, in which case we require that $u_k\in M_{d_k}(x_{k-1},x_k,p_{x_k})$. On each of these cylic fiber products, there is an evaluation map $e_+$ with values in $S(x_0)$, and an evaluation map $e_-$ with values in $S(x_k)$.

Also, we define $M_d^*(x_+,x_-)$ to be the set of $u\in M_d(x_+,x_-)$ such that $e_+(u)\neq p_{x_+}$ and $e_-(u)\neq p_{x_-}$. Similarly, we define $M_d^*(x_+,p_+,x_-)$ to be the set of $u\in M_d(x_+,p_+,x_-)$ such that $e_-(u)\neq p_-$; and we define $M_d^*(x_+,x_-,p_-)$ to be the set of $u\in M_d(x_+,x_-,p_{x_-})$ such that $e_+(u)\neq p_{x_+}$.

For $x_+,x_-\in X$ and $d\in\{0,1\}$, we now define four cascade moduli spaces $M^\ca_d(\widehat{x}_+,\widehat{x}_-)$, $M^\ca_d(\widehat{x}_+,\widecheck{x}_-)$, $M^\ca_d(\widecheck{x}_+,\widehat{x}_-)$, and $M^\ca_d(\widecheck{x}_+,\widecheck{x}_-)$. These will be open smooth manifolds of dimension $d$, with orientations valued in $\mc{O}_{x_+}(p_{x_+})\tensor \mc{O}_{x_-}(p_{x_-})$.

The simplest case is where $x_+=x_-$. In this case we define
\begin{equation}
\label{eqn:cascadesimple1}
M^\ca_d(\widehat{x},\widehat{x}) = M^\ca_d(\widecheck{x},\widehat{x}) = M^\ca_d(\widecheck{x},\widecheck{x}) = \emptyset
\end{equation}
and
\begin{equation}
\label{eqn:cascadesimple2}
M^\ca_d(\widehat{x},\widecheck{x}) = \left\{\begin{array}{cl} \{S_x,p_x\}, & \mbox{if $d=0$},\\
\emptyset & \mbox{otherwise}.
\end{array}
\right.
\end{equation}
That is, the set $M^\ca_0(\widehat{x},\widecheck{x})$ has two elements, which we label as $S_x$ and $p_x$. To complete this definition, we need to specify the orientation of $M^\ca_0(\widehat{x},\widecheck{x})$ with values in
\[
\mc{O}_x(p_x) \tensor \mc{O}_x(p_x) = \Z.
\]
That is, we need to attach a sign to each of the two points $S_x$ and $p_x$.  If the local system $\mc{O}_x$ is trivial, then $S_x$ has positive sign and $p_x$ has negative sign.  If the local system $\mc{O}_x$ is not trivial, then $S_x$ and $p_x$ both have negative signs.

\begin{remark}
The above convention is the reason for the $-2$ in equation \eqref{eqn:minustwo}. Combined with the rather natural orientation conventions below, the above convention is necessary for the orientations to work out in Proposition~\ref{prop:cascadekey}(b) below.
\end{remark}

We now define the cascade moduli spaces when $x_+\neq x_-$.

\subsubsection{Unconstrained cascade moduli spaces}

If $x_+\neq x_-$, we define
\begin{equation}
\label{eqn:cascadehatcheck}
\begin{split}
M^\ca_d(\widehat{x}_+,\widecheck{x}_-) =& \coprod_{k\ge 1} \coprod_{x_+=x_0,x_1,\ldots,x_k=x_-} \coprod_{d_1+\cdots +d_k=d}
\\
&M^*_{d_1}(x_0,x_1) \underset{S(x_1)}{\circlearrowleft} M^*_{d_2}(x_1,x_2) \underset{S(x_2)}{\circlearrowleft} \cdots \underset{S(x_{k-1})}{\circlearrowleft} M^*_{d_k}(x_{k-1},x_k).
\end{split}
\end{equation}

The orientation of $M^\ca_d(\widehat{x}_+,\widecheck{x}_-)$ with values in $\mc{O}_{x_+}(p_{x_+})\tensor \mc{O}_{x_-}(p_{x_-})$ is defined as follows. Consider a point
\[
\begin{split}
(u_1,\ldots,u_k) \in& M^*_{d_1}(x_0,x_1) \underset{S(x_1)}{\circlearrowleft} M^*_{d_2}(x_1,x_2) \underset{S(x_2)}{\circlearrowleft} \cdots \underset{S(x_{k-1})}{\circlearrowleft} M^*_{d_k}(x_{k-1},x_k)\\
&\subset M^\ca_d(\widehat{x}_+,\widecheck{x}_-).
\end{split}
\]
At $u_i$, the moduli space $M_{d_i}(x_{i-1},x_i)$ {\em a priori\/} has an orientation with values in $\mc{O}_{x_{i-1}}(e_+(u_i)) \tensor \mc{O}_{x_i}(e_-(u_i))$. The cyclic fiber product is an open subset of the product of these moduli spaces for $i=1,\ldots,k$. Thus taking the product of these orientations in order from $i=1$ to $k$, we obtain an orientation of the cascade moduli space $M^\ca_d(\widehat{x}_+,\widecheck{x}_-)$ at the point $(u_1,\ldots,u_k)$ with values in
\begin{equation}
\label{eqn:apo}
\bigotimes_{i=1}^k \mc{O}_{x_{i-1}}(e_+(u_i)) \tensor \mc{O}_{x_i}(e_-(u_i)).
\end{equation}
Now there is an isomorphism
\begin{equation}
\label{eqn:pt1}
\mc{O}_{x_+}(p_{x_+}) = \mc{O}_{x_0}(p_{x_0})\simeq \mc{O}_{x_0}(e_+(u_1))
\end{equation}
obtained by parallel transport in $\mc{O}_{x_0}$ along a positively oriented path in $S(x_0)$ from $p_{x_0}$ to $e_+(u_1)$. Similarly, for $i=1,\ldots,k=1$ we have an isomorphism
\begin{equation}
\label{eqn:pt2}
\mc{O}_{x_i}(e_-(u_i)) \simeq \mc{O}_{x_i}(e_+(u_{i+1}))
\end{equation}
obtained by parallel transport in $\mc{O}_{x_i}$ along a positively oriented path in $S(x_i)$ from $e_-(u_i)$ to $e_+(u_{i+1})$.
Finally, there is an isomorphism
\begin{equation}
\label{eqn:pt3}
\mc{O}_{x_k}(e_-(u_k))\simeq \mc{O}_{x_k}(p_{x_k}) = \mc{O}_{x_-}(p_{x_-})
\end{equation}
obtained by parallel transport in $\mc{O}_{x_k}$ along a positively oriented path in $S(x_k)$ from $e_-(u_k)$ to $p_{x_k}$.  Combining the isomorphisms \eqref{eqn:pt1}, \eqref{eqn:pt2} and \eqref{eqn:pt3} allows us to identify \eqref{eqn:apo} as
\[
\bigotimes_{i=1}^k \mc{O}_{x_{i-1}}(e_+(u_i)) \tensor \mc{O}_{x_i}(e_-(u_i)) \simeq \mc{O}_{x_+}(p_{x_+}) \tensor \mc{O}_{x_-}(p_{x_-}).
\]
Thus we obtain an orientation of $M^\ca_d(\widehat{x}_+,\widecheck{x}_-)$ at the point $(u_1,\ldots,u_k)$ with values in $\mc{O}_{x_0}(p_{x_0}) \tensor \mc{O}_{x_k}(p_{x_k})$. As $(u_1,\ldots,u_k)$ moves in $M^\ca_d(\widehat{x}_+,\widehat{x}_-)$, this orientation is continuous, because:
\begin{itemize}
\item
Since we are using $M_{d_1}^*(x_0,x_1)$ in \eqref{eqn:cascadehatcheck}, so that $p_{x_0}\neq e_+(u_1)$, the isomorphism \eqref{eqn:pt1} varies continuously.
\item
For $i=1,\ldots, k=1$, since we use the ``cyclic fiber product'' condition $\underset{S(x_i)}{\circlearrowleft}$ in \eqref{eqn:cascadehatcheck}, so that $e_-(u_i)\neq e_+(u_{i+1})$, the isomorphism \eqref{eqn:pt2} varies continuously.
\item
Since we are using $M_{d_k}^*(x_{k-1},x_k)$ in \eqref{eqn:cascadehatcheck}, so that $e_-(u_k)\neq p_{x_k}$, the isomorphism \eqref{eqn:pt3} varies continuously.
\end{itemize}

\subsubsection{Constrained cascade moduli spaces}

Let $x_+,x_-\in X$ be distinct. We define
\begin{equation}
\label{eqn:cascadecheckcheck}
\begin{split}
M^\ca_d(\widecheck{x}_+,\widecheck{x}_-) =& \coprod_{k\ge 1} \coprod_{x_+=x_0,x_1,\ldots,x_k=x_-} \coprod_{d_1+\cdots +d_k=d+1}
\\
&M_{d_1}^*(x_0,p_{x_0},x_1) \underset{S(x_1)}{\circlearrowleft} M^*_{d_2}(x_1,x_2) \underset{S(x_2)}{\circlearrowleft} \cdots \underset{S(x_{k-1})}{\circlearrowleft} M^*_{d_k}(x_{k-1},x_k).
\end{split}
\end{equation}
This is oriented as in \eqref{eqn:cascadehatcheck}, except that now we can dispense with the isomorphism \eqref{eqn:pt1}. We define
\begin{equation}
\label{eqn:cascadehathat}
\begin{split}
M^\ca_d(\widehat{x}_+,\widehat{x}_-) =& \coprod_{k\ge 1} \coprod_{x_+=x_0,x_1,\ldots,x_k=x_-} \coprod_{d_1+\cdots +d_k=d+1}
\\
&M^*_{d_1}(x_0,x_1) \underset{S(x_1)}{\circlearrowleft} M_{d_2}^*(x_1,x_2) \underset{S(x_2)}{\circlearrowleft} \cdots \underset{S(x_{k-1})}{\circlearrowleft} M_{d_k}^*(x_{k-1},x_k,p_{x_k}).
\end{split}
\end{equation}
This is oriented as in \eqref{eqn:cascadehatcheck}, except that here we can dispense with the isomorphism \eqref{eqn:pt3}. Finally, we define
\begin{equation}
\label{eqn:cascadecheckhat}
\begin{split}
M^\ca_d(\widecheck{x}_+,\widehat{x}_-) = &
M_{d+2}(x_+,p_{x_+},x_-,p_{x_-})
\\
& \bigsqcup \coprod_{k\ge 2} \coprod_{x_+=x_0,x_1,\ldots,x_k=x_-} \coprod_{d_1+\cdots +d_k=d+2}
\\
&M_{d_1}^*(x_0,p_{x_0},x_1) 
\underset{S(x_1)}{\circlearrowleft} M_{d_2}^*(x_1,x_2) \underset{S(x_2)}{\circlearrowleft} \cdots \underset{S(x_{k-1})}{\circlearrowleft}
 M_{d_k}^*(x_{k-1},x_k,p_{x_k}).
\end{split}
\end{equation}
This is oriented as in \eqref{eqn:cascadehatcheck}, except that now we can dispense with the isomorphisms \eqref{eqn:pt1} and \eqref{eqn:pt3}.

\subsection{Cascade moduli spaces: the key property}

Recall from \S\ref{sec:cascadesetup} that $\mc{P}$ denotes the set of choices $\{p_x\}_{x\in X}$.

\begin{proposition}
\label{prop:cascadekey}
Let $x_+,x_-\in X$. Fix $\widetilde{x}_+$ to denote one of $\widehat{x}_+$ or $\widecheck{x}_+$; and fix $\widetilde{x}_-$ to denote one of $\widehat{x}_-$ or $\widecheck{x}_-$. If $\mc{P}$ is generic then:
\begin{description}
\item{(a)}
The cascade moduli space $M^\ca_0(\widetilde{x}_+,\widetilde{x}_-)$
is finite.
\item{(b)}
The cascade moduli space $M^\ca_1(\widetilde{x}_+,\widetilde{x}_-)$ has a compactification $\overline{M}^\ca_1(\widetilde{x}_+,\widetilde{x}_-)$ with boundary
\begin{equation}
\label{eqn:caco}
\begin{split}
\partial \overline{M}^\ca_1(\widetilde{x}_+,\widetilde{x}_-) = & \coprod_{y\in X}M^\ca_0(\widetilde{x}_+,\widecheck{y}) \times M^\ca_0(\widecheck{y},\widetilde{x}_-)\\
&\bigsqcup \coprod_{y\in X} M^\ca_0(\widetilde{x}_+,\widehat{y}) \times M^\ca_0(\widehat{y},\widetilde{x}_-).
\end{split}
\end{equation}
In \eqref{eqn:caco}, the boundary orientation on the left hand side agrees with the product orientation on the right hand side.
\end{description}
\end{proposition}

\begin{proof}
(a) By \eqref{eqn:cascadesimple1} and \eqref{eqn:cascadesimple2}, we can assume that $x_+\neq x_-$.

In the unconstrained case when $\widetilde{x}_+=\widehat{x}_+$ and $\widetilde{x}_-=\widecheck{x}_-$, so that all factors in the cascade \eqref{eqn:cascadehatcheck} live in $0$-dimensional moduli spaces $M_0(x_{i-1},x_i)$, the desired finiteness of $M_0^\ca(\widehat{x}_+,\widecheck{x}_-)$ follows from the Finiteness axiom and part (a) of the Compactification axiom.

If instead we have $\widetilde{x}_+=\widecheck{x}_+$, then we also need to know that $M_1(x_+,p_{x_+},x)$ is finite for every $x\neq x_+$. Suppose to get a contradiction that there is an infinite sequence $\{u_i\}_{i=1,\ldots}$ of distinct elements of $M_1(x_+,p_{x_+},x)$. By part (b) of the Compactification axiom, we can pass to a subsequence so that $\{u_i\}$ converges to a point $u_\infty\in \overline{M}_1(x_+,x)$ with $e_+(u_\infty)=p_{x_+}$. If $u_\infty\in M_1(x_+,x)$, this contradicts the assumption that $p_{x_+}$ is a regular value of $e_+$.  Thus $u_\infty\in\overline{M}_1(x_+,x)\setminus M_1(x_+,x)$. Since the latter set is finite, if $p_{x_+}$ is generic then it is not in $e_+$ of this set, which is also a contradiction. Therefore, if $\mc{P}$ is generic then $M_1(x_+,p_{x_+},x)$ is finite for every $x\neq x_+$, and consequently $M_0^\ca(\widecheck{x}_+,\widecheck{x}_-)$ is finite.

Similar arguments show that if $\mc{P}$ is generic then $M_1(x,x_-,p_{x_-})$ is finite for every $x\neq x_-$, and $M_2(x_+,p_{x_+},x_-,p_{x_-})$ is finite. We then likewise deduce that $M_0^\ca(\widehat{x}_+,\widehat{x}_-)$ and $M_0^\ca(\widecheck{x}_+,\widehat{x}_-)$ are finite.

(b) If $x_+=x_-$, then $M^\ca_1(\widetilde{x}_+,\widetilde{x}_-)=\emptyset$ by definition, so we can (and must) take the compactification to be the empty set. The right hand side of \eqref{eqn:caco} is also empty; otherwise we could make arbitrarily long chains of nonempty moduli spaces, violating the Finiteness axiom.

Suppose now that $x_+\neq x_-$. There are four cases to consider, depending on whether $\widetilde{x}_+$ equals $\widehat{x}_+$ or $\widecheck{x}_+$, and whether $\widetilde{x}_-$ equals $\widehat{x}_-$ or $\widecheck{x}_-$. We will just consider the case where $\widetilde{x}_+=\widehat{x}_+$ and $\widetilde{x}_-=\widehat{x}_-$; the proofs in the other cases use the same ideas. We now need to show that $\M^\ca_1(\widehat{x}_+,\widehat{x}_-)$ has a compactification $\overline{M}^\ca_1(\widehat{x}_+,\widehat{x}_-)$ with oriented boundary
\begin{equation}
\label{eqn:nnts}
\begin{split}
\partial \overline{M}^\ca_1(\widehat{x}_+,\widehat{x}_-) = & \coprod_{y\in X}M^\ca_0(\widehat{x}_+,\widecheck{y}) \times M^\ca_0(\widecheck{y},\widehat{x}_-)\\
&\bigsqcup \coprod_{y\in X} M^\ca_0(\widehat{x}_+,\widehat{y}) \times M^\ca_0(\widehat{y},\widehat{x}_-).
\end{split}
\end{equation}

Recall that
\begin{equation}
\label{eqn:recallmca1}
\begin{split}
{M^\ca_1}(\widehat{x}_+,\widehat{x}_-) =& \coprod_{k\ge 1} \coprod_{x_+=x_0,x_1,\ldots,x_k=x_-} \coprod_{d_1+\cdots +d_k=2}
\\
&{M}_{d_1}^*(x_0,x_1) \underset{S(x_1)}{\circlearrowleft} {M}_{d_2}^*(x_1,x_2) \underset{S(x_2)}{\circlearrowleft} \cdots \underset{S(x_{k-1})}{\circlearrowleft} {M}_{d_k}^*(x_{k-1},x_k,p_{x_k}).
\end{split}
\end{equation}
Note that if $(u_1,\ldots,u_k)$ is in the above moduli space, then only one of the factors $u_i$ is in a $1$-dimensional moduli space; this is $M_1^*(x_{i-1},x_{i})$ if $i<k$, and $M_2^*(x_{k-1},x_k,p_{x_k})$ if $i=k$. Every other $u_i$ is rigid, i.e.\ in a $0$-dimensional moduli space; this is $M_0^*(x_{i-1},x_i)$ if $i<k$ and $M_1^*(x_{k-1},x_k,p_{x_k})$ if $i=k$.

The idea of the proof of \eqref{eqn:nnts} is that the moduli space \eqref{eqn:recallmca1} has ends where one of the following happens: (i) the non-rigid $u_i$ approaches an end of its moduli space; (ii) $e_+(u_i)$ approaches $e_-(u_{i-1})$ (when $i>1$); (iii) $e_-(u_i)$ approaches $e_+(u_{i+1})$ (when $i<k$); (iv) $e_+(u_i)$ approaches $p_{x_{i-1}}$; or (v) $e_-(u_i)$ approaches $p_{x_i}$ (when $i<k$). We can compactify the moduli space by gluing together ends of the form (i), (ii), and (iii), and adding boundary points for ends of the form (iv) and (v). Boundary points of type (iv) correspond to the first line on the right hand side of \eqref{eqn:nnts}, and boundary points of type (v) correspond to the second line of \eqref{eqn:nnts}.

To be more precise, and to explain how the orientations work, note that there are four possibilities for the $i$ such that $u_i$ is not rigid: $1=i=k$, $1=i<k$, $1<i=k$, or $1<i<k$. Accordingly, we can write
\[
\begin{split}
{M^\ca_1}(\widehat{x}_+,\widehat{x}_-) =& M_2^*(x_+,x_-,p_{x_-})\\
&\bigsqcup \coprod_{x_+\neq x'\neq x_-} M_1^*(x_+,x') \underset{S(x')}{\circlearrowleft} M^\ca_0(\widehat{x}',\widehat{x}_-)\\
&\bigsqcup \coprod_{x_+\neq x'\neq x_-} M^\ca_0(\widehat{x}_+,\widecheck{x}') \underset{S(x')}{\circlearrowleft} M_2^*(x',x_-,p_{x_-})\\
&\bigsqcup \coprod_{x_+\neq x'\neq x''\neq x_-} M^\ca_0(\widehat{x}_+,\widecheck{x}') \underset{S(x')}{\circlearrowleft} M_1^*(x',x'') \underset{S(x'')}{\circlearrowleft} M^\ca_0(\widehat{x}'',\widehat{x}_-).
\end{split}
\]

We first define a ``partial compactification'' of $M^\ca_1(\widehat{x}_+,\widehat{x}_-)$ by compactifying the $1$-dimensional moduli spaces above, to get
\[
\widetilde{M}^\ca_1(\widehat{x}_+,\widehat{x}_-) = F_1 \sqcup F_2 \sqcup F_3 \sqcup F_4
\]
where
\[
\begin{split}
F_1 &= \overline{M}_2^*(x_+,x_-,p_{x_-}),\\
F_2 &= \coprod_{x_+\neq x'\neq x_-} \overline{M}_1^*(x_+,x') \underset{S(x')}{\circlearrowleft} M^\ca_0(\widehat{x}',\widehat{x}_-),\\
F_3 &= \coprod_{x_+\neq x'\neq x_-} M^\ca_0(\widehat{x}_+,\widecheck{x}') \underset{S(x')}{\circlearrowleft} \overline{M}_2^*(x',x_-,p_{x_-}),\\
F_4 &= \coprod_{x_+\neq x'\neq x'' \neq x_-} M^\ca_0(\widehat{x}_+,\widecheck{x}') \underset{S(x')}{\circlearrowleft} \overline{M}_1^*(x',x'') \underset{S(x'')}{\circlearrowleft} M^\ca_0(\widehat{x}'',\widehat{x}_-).
\end{split}
\]
Here $\overline{M}_2^*(x_+,x_-,p_{x_-})$ denotes the set of $u\in\overline{M}_2(x_+,x_-,p_{x_-})$ such that $e_+(u)\neq p_{x_+}$, and so forth; and the cyclic fiber products are oriented as before.

By parts (b) and (c) of the Compactification axiom, the oriented boundaries of the four parts of the partial compactification are given by
\[
\begin{split}
\partial F_1 =&
\coprod_{x_+\neq x'\neq x_-}M_0^*(x_+,x')\times_{S(x')}M_2^*(x',x_-,p_{x_-}) \\
& - \coprod_{x_+\neq x'\neq x_-}M_1^*(x_+,x')\times_{S(x')} M_1^*(x',x_-,p_{x_-}),
\\
\partial F_2 =& \coprod_{x_+\neq x'\neq x''\neq x_-} {M}_0^*(x_+,x')\times_{S(x')}M_1^*(x',x'') \underset{S(x'')}{\circlearrowleft} M^\ca_0(\widehat{x}'',\widehat{x}_-)\\
& -\coprod_{x_+\neq x'\neq x''\neq x_-} {M}_1^*(x_+,x')\times_{S(x')}M_0^*(x',x'') \underset{S(x'')}{\circlearrowleft} M^\ca_0(\widehat{x}'',\widehat{x}_-),
\\
\partial F_3 =& \coprod_{x_+\neq x'\neq x''\neq x_-} M^\ca_0(\widehat{x}_+,\widecheck{x}') \underset{S(x')}{\circlearrowleft} M_0^*(x',x'')\times_{S(x'')}{M}_2^*(x'',x_-,p_{x_-})\\
& -\coprod_{x_+\neq x'\neq x''\neq x_-} M^\ca_0(\widehat{x}_+,\widecheck{x}') \underset{S(x')}{\circlearrowleft} M_1^*(x',x'')\times_{S(x'')}{M}_1^*(x'',x_-,p_{x_-}),
\\
\partial F_4 =& \coprod_{x_+\neq x'\neq x''\neq x'''\neq x_-} M^\ca_0(\widehat{x}_+,\widecheck{x}') \underset{S(x')}{\circlearrowleft} {M}_0^*(x',x'')\times_{S(x'')}M_1^*(x'',x''') \underset{S(x''')}{\circlearrowleft} M^\ca_0(\widehat{x}''',\widehat{x}_-)\\
& -\coprod_{x_+\neq x'\neq x''\neq x'''\neq x_-} M^\ca_0(\widehat{x}_+,\widecheck{x}') \underset{S(x')}{\circlearrowleft} {M}_1^*(x',x'')\times_{S(x'')}M_0^*(x'',x''') \underset{S(x''')}{\circlearrowleft} M^\ca_0(\widehat{x}''',\widehat{x}_-).
\end{split}
\]
Note that we can use starred moduli spaces in the above equations, by our assumption that each point $p_x$ is a regular value of all evaluation maps $e_\pm$.

We can combine the above four equations to obtain the following formula for the boundary of the partial compactification:
\begin{equation}
\label{eqn:bpc}
\begin{split}
\partial\widetilde{M}^\ca_1(\widehat{x}_+,\widehat{x}_-) =& \coprod_{x_+\neq x'\neq x_-} M_0^\ca(\widehat{x}_+,\widecheck{x}')\times_{S(x')} M^\ca_1(\widehat{x}',\widehat{x}_-)\\
& - \coprod_{x_+\neq x'\neq x_-} M_1^\ca(\widehat{x}_+,\widecheck{x}') \times_{S(x')} M^\ca_0(\widehat{x}',\widehat{x}_-).
\end{split}
\end{equation}
Here the first line of \eqref{eqn:bpc} corresponds to the first lines of the previous four equations.

The partial compactification also has ends, where $e_+$ or $e_-$ of the non-rigid factor approaches a forbidden value. We now classify these. Here the signs are determined by the orientation conventions in \S\ref{sec:conventions} and Convention~\ref{convention:constraints}.

To start, there are ends of $F_1$ where $e_+$ approaches $p_{x_+}$ from either side. Thus
\begin{equation}
\label{eqn:endsf1}
\op{Ends}(F_1) = M^\ca_0(\widehat{x}_+,\widecheck{x}_+)\times M_2(x_+,p_{x_+},x_-,p_{x_-}).
\end{equation}
Note that since we are assuming that the points $p_{x_+}$ and $p_{x_-}$ are generic, we do not need a bar on $M_2$ in this equation.

Next, there are ends of $F_2$ where $e_+$ of the first factor approaches $p_{x_+}$, where $e_-$ of the first factor approaches $p_{x'}$, and where $e_-$ of the first factor approaches $e_+$ of the second factor. Thus
\begin{equation}
\label{eqn:endsf2}
\begin{split}
\op{Ends}(F_2) =& \coprod_{x_+\neq x'\neq x_-} M^\ca_0(\widehat{x}_+,\widecheck{x}_+) \times M_1^*(x_+,p_{x_+},x') \underset{S(x')}{\circlearrowleft} M^\ca_0(\widehat{x}',\widehat{x}_-)\\
& + \coprod_{x_+\neq x'\neq x_-} M_1^*(x_+,x',p_{x'}) \times M^\ca_0(\widehat{x}',\widehat{x}_-)\\
& + \coprod_{x_+\neq x'\neq x_-} M_1^*(x_+,x')\times_{S(x')} M^\ca_0(\widehat{x}',\widehat{x}_-).
\end{split}
\end{equation}
Note that we do not need bars on $M_1$ in the first two lines of this equation because $p_{x_+}$ and $p_{x_-}$ are generic; and we do not need a bar on $M_1$ in the third line by part (e) of the Compactification axiom.

Next, there are ends of $F_3$ where $e_+$ of the second factor approaches $p_{x'}$, and where $e_+$ of the second factor approaches $e_-$ of the first factor. Thus
\begin{equation}
\label{eqn:endsf3}
\begin{split}
\op{Ends}(F_3) =& \coprod_{x_+\neq x'\neq x_-} M^\ca_0(\widehat{x}_+,\widecheck{x}') \times M_2(x',p_{x'},x_-,p_{x_-})\\
& - \coprod_{x_+\neq x'\neq x_-} M^\ca_0(\widehat{x}_+,\widecheck{x}') \times_{S(x')} M_2^*(x',x_-,p_{x_-}),
\end{split}
\end{equation}
where we do not need bars on $M_2$ as before.

Finally, there are ends of $F_4$ where $e_+$ of the second factor approaches $p_{x'}$, where $e_+$ of the second factor approaches $e_-$ of the first factor, where $e_-$ of the second factor approaches $p_{x''}$, and where $e_-$ of the second factor approaches $e_+$ of the third factor. Similarly to the above, we obtain
\begin{equation}
\label{eqn:endsf4}
\begin{split}
\op{Ends}(F_4) =& \coprod_{x_+\neq x'\neq x''\neq x_+} M^\ca_0(\widehat{x}_+,\widecheck{x}') \times {M}_1^*(x',p_{x'},x'') \underset{S(x'')}{\circlearrowleft} M^\ca_0(\widehat{x}'',\widehat{x}_-)\\
& -\coprod_{x_+\neq x'\neq x''\neq x_-} M^\ca_0(\widehat{x}_+,\widecheck{x}') \times_{S(x')} {M}_1^*(x',x'') \underset{S(x'')}{\circlearrowleft} M^\ca_0(\widehat{x}'',\widehat{x}_-)\\
& +\coprod_{x_+\neq x'\neq x''\neq x_-} M^\ca_0(\widehat{x}_+,\widecheck{x}') \underset{S(x')}{\circlearrowleft} {M}_1^*(x',x'',p_{x''}) \times M^\ca_0(\widehat{x}'',\widehat{x}_-)\\
& +\coprod_{x_+\neq x'\neq x''\neq x_-} M^\ca_0(\widehat{x}_+,\widecheck{x}') \underset{S(x')}{\circlearrowleft} {M}_1^*(x',x'') \times_{S(x'')} M^\ca_0(\widehat{x}'',\widehat{x}_-).
\end{split}
\end{equation}

Putting the above four equations together, we obtain
\begin{equation}
\label{eqn:endspc}
\begin{split}
\op{Ends}(\widetilde{M}^\ca_1(\widehat{x}_+,\widehat{x}_-)) =& -\coprod_{x_+\neq x'\neq x_-} M_0^\ca(\widehat{x}_+,\widecheck{x}')\times_{S(x')} M^\ca_1(\widehat{x}',\widehat{x}_-)\\
& + \coprod_{x_+\neq x'\neq x_-} M_1^\ca(\widehat{x}_+,\widecheck{x}') \times_{S(x')} M^\ca_0(\widehat{x'},\widehat{x}_-)\\
& +\coprod_{y}M^\ca_0(\widehat{x}_+,\widecheck{y}) \times M^\ca_0(\widecheck{y},\widehat{x}_-)\\
& +\coprod_{y} M^\ca_0(\widehat{x}_+,\widehat{y}) \times M^\ca_0(\widehat{y},\widehat{x}_-).
\end{split}
\end{equation}
Here the first line of \eqref{eqn:endspc} corresponds to the second lines of \eqref{eqn:endsf3} and \eqref{eqn:endsf4}; the second line of \eqref{eqn:endspc} corresponds to the third line of \eqref{eqn:endsf2} and the fourth line of \eqref{eqn:endsf4}; the third line of \eqref{eqn:endspc} corresponds to the first lines of \eqref{eqn:endsf1}, \eqref{eqn:endsf2}, \eqref{eqn:endsf3}, and \eqref{eqn:endsf4}; and the fourth line of \eqref{eqn:endspc} corresponds to the second line of \eqref{eqn:endsf2} and the third line of \eqref{eqn:endsf4}.

To conclude, the first two lines of \eqref{eqn:endspc} match the corresponding lines of \eqref{eqn:bpc}, but with opposite orientations. Thus we can glue these ends and boundary points together, and add points corresponding to the last two lines of \eqref{eqn:endspc}, to obtain the desired compactification of $M^\ca_1(\widehat{x}_+,\widehat{x}_-)$ satisfying \eqref{eqn:nnts}.
\end{proof}

\subsection{Definition of cascade homology}

We already explained in \S\ref{sec:cascadesetup} how to define the $\Z/2$-graded chain module $C_*^\ca(A,\mc{P})$.

\begin{definition}
We define the differential
\[
\partial: C_*^\ca(A,\mc{P}) \longrightarrow C_{*-1}^\ca(A,\mc{P})
\]
as follows. Let $x\in X$, and fix $\widetilde{x}$ to denote one of $\widehat{x}$ or $\widecheck{x}$. Define
\[
\partial\widetilde{x} = \sum_{y\in X} \left( \#M_0^\ca\left(\widetilde{x},\widecheck{y}\right) \widecheck{y} + \#M_0^\ca\left(\widetilde{x},\widehat{y}\right) \widehat{y}\right).
\]
\end{definition}

Here $\#M_0^\ca$ denotes the signed count of points in $M_0^\ca$; it follows from Proposition~\ref{prop:cascadekey}(a) that this is well defined. Furthermore, the Finiteness axiom guarantees that the whole sum is finite. It follows from the Grading axiom that $\partial$ decreases the mod $2$ grading by $1$. Finally, Proposition~\ref{prop:cascadekey}(b) shows that $\partial^2=0$.

\begin{definition}
\label{def:hca}
We define the cascade homology $H_*^\ca(A,\mc{P})$ to be the homology of the chain complex $\left(C_*^\ca(A,\mc{P}),\partial\right)$.
\end{definition}

\subsection{The conjugate of a Morse-Bott system}
\label{sec:conjugate}

To clarify some signs in the definition of induced maps on cascade homology, it will help to consider a modification of a Morse-Bott system in which the orientation on each moduli space $M_d$ is multiplied by $(-1)^d$.

\begin{definition}
If $A=(X,|\cdot|,S,\mc{O},M_*,e_\pm)$ is a Morse-Bott system, define its {\em conjugate\/} $\underline{A}=(X,|\cdot|,S,\mc{O},\underline{M}_*,e_\pm)$, where 
\[
\underline{M}_d(x_+,x_-) = (-1)^d M_d(x_+,x_-).
\]
\end{definition}

Note that the conjugate of a Morse-Bott system is also a Morse-Bott system, because when we pass to the conjugate, for each of the equations \eqref{eqn:compactification1}--\eqref{eqn:compactification3}, both sides change sign in the same way. Conjugation also does not affect the cascade homology: It follows from the Grading axiom that the chain complexes $C^\ca_*(\underline{A},\mc{P})$ and $C^\ca_*(A,\mc{P})$ are isomorphic via the involution which multiplies each generator $\widecheck{x}$ or $\widehat{x}$ by $(-1)^{|x|}$.

For our purposes, a slightly different involution will be more useful:

\begin{definition}
Define
\[
\tau: C^\ca_*(\underline{A},\mc{P}) \longrightarrow C^\ca_*(A,\mc{P})
\]
by $\tau(\widecheck{x})=\widecheck{x}$ and $\tau(\widehat{x})=-\widehat{x}$.
\end{definition}

\begin{lemma}
\label{lem:tauacm}
Let $\underline{\partial}$ denote the differential on $C^\ca_*(\underline{A},\mc{P})$. Then
\[
\partial\tau = -\tau\underline{\partial}.
\]
\end{lemma}

\begin{proof}
By \eqref{eqn:cascadehatcheck} and \eqref{eqn:cascadecheckhat}, the differentials from hat generators to check generators or vice-versa count cascades with total moduli space dimension even. By \eqref{eqn:cascadecheckcheck} and \eqref{eqn:cascadehathat}, the differentials between check and check generators, or between hat and hat generators, count cascades with total moduli space dimension odd.
\end{proof}

The reason why conjugation is useful is that if $\Phi:A_1\to A_2$ is a morphism of Morse-Bott systems, then we can rewrite equations \eqref{eqn:compactiphi1}--\eqref{eqn:compactiphi3} using the conjugate of $A_1$ (but not the conjugate of $A_2$) to obtain nicer signs, which look just like the signs in equations \eqref{eqn:compactification1}--\eqref{eqn:compactification3}. Namely, we have
\begin{equation}
\label{eqn:conjugatemorphism}
\begin{split}
\partial \overline{\Phi}_1(x_1,x_2) =& \coprod_{\substack{x_1'\in X_1\setminus\{x_1\}\\d_1+d=1}}
(-1)^{d_1} \underline{M}_{d_1}^1(x_1,x_1') \times_{S_1(x_1')} \Phi_d(x_1',x_2)\\
& \bigsqcup \coprod_{\substack{x_2'\in X_2\setminus\{x_2\}\\ d+d_2=1}} 
(-1)^{d}
\Phi_d(x_1,x_2') \times_{S_2(x_2')}M_{d_2}^2(x_2',x_2),\\
\partial \overline{\Phi}_2(x_1,x_2,p_2) =& \coprod_{\substack{x_1'\in X_1\setminus\{x_1\}\\d_1+d=2}}
(-1)^{d_1}\underline{M}_{d_1}^1(x_1,x_1') \times_{S_1(x_1')} \Phi_d(x_1',x_2,p_2)\\
& \bigsqcup \coprod_{\substack{x_2'\in X_2\setminus\{x_2\}\\ d+d_2=2}} 
(-1)^{d}
\Phi_d(x_1,x_2') \times_{S_2(x_2')} M_{d_2}^2(x_2',x_2,p_2),\\
\partial \overline{\Phi}_2(x_1,p_1,x_2) =& \coprod_{\substack{x_1'\in X_1\setminus\{x_1\}\\d_1+d=2}}
(-1)^{d_1-1}
\underline{M}_{d_1}^1(x_1,p_1,x_1') \times_{S_1(x_1')} \Phi_d(x_1',x_2)\\
& \bigsqcup \coprod_{\substack{x_2'\in X_2\setminus\{x_2\}\\ d+d_2=2}} 
(-1)^{d-1}
\Phi_d(x_1,p_1,x_2') \times_{S_2(x_2')} M_{d_2}^2(x_2',x_2),\\
\partial \overline{\Phi}_3(x_1,p_1,x_2,p_2) =& \coprod_{\substack{x_1'\in X_1\setminus\{x_1\}\\d_1+d=3}}
(-1)^{d_1-1}
\underline{M}_{d_1}^1(x_1,p_1,x_1') \times_{S_1(x_1')} \Phi_d(x_1',x_2,p_2)\\
& \bigsqcup \coprod_{\substack{x_2'\in X_2\setminus\{x_2\}\\ d+d_2=3}} 
(-1)^{d-1}
\Phi_d(x_1,p_1,x_2') \times_{S_2(x_2')} M_{d_2}^2(x_2',x_2,p_2).
\end{split}
\end{equation}

\subsection{Induced maps on cascade homology}
\label{sec:inducedmap}

Let $A_1=(X_1,|\cdot|_1,S_1,\mc{O}^1,M^1_*,e^1_\pm)$ and $A_2=(X_2,|\cdot|_2,S_2,\mc{O}^2,M^2_*,e^2_\pm)$ be Morse-Bott systems, and let $\Phi:A_1\to A_2$ be a morphism of Morse-Bott systems. Let $(\mc{P}_1,\mc{P}_2)$ be a generic pair of choices as needed to define the cascade chain complexes $C^\ca_*(A_1,\mc{P}_1)$ and $C^\ca_*(A_2,\mc{P}_2)$. We now define a chain map
\[
\Phi_\sharp: C^\ca_*(A_1,\mc{P}_1) \longrightarrow C^\ca_*(A_2,\mc{P}_2).
\]

The idea is to define $\Phi_\sharp$ by counting ``hybrid'' cascades consisting of some elements of $M^1_d$, followed by an element of $\Phi_d$, followed by some elements of $M^2_d$, with total moduli space dimension zero (after point constraints are taken into account). The chain map equation arises by considering such cascades with total moduli space dimension one. In order to simplify the notation when defining this precisely, we will use the following shortcut.

\begin{definition}
\label{def:A1A2}
Define a Morse-Bott system
\[
\underline{A}_1 \sqcup_\Phi A_2 = (X,|\cdot|,S,\mc{O},M_*,e_\pm)
\]
as follows. We take $X=X_1\sqcup X_2$. For $x_1\in X_1$ we define $|x_1|=|x_1|_1+1$, $S_{x_1}=S^1_{x_1}$, and $\mc{O}_{x_1} = \mc{O}^1_{x_1}$. For $x_2\in X_2$ we define $|x_2|=|x_2|_2$, $S_{x_2}=S^2_{x_2}$, and $\mc{O}_{x_2} = \mc{O}^2_{x_2}$. For $x_1,x_1'\in X_1$ and $x_2,x_2'\in X_2$, we define
\[
\begin{split}
M_d(x_1,x_1') &= \underline{M}^1_d(x_1,x_1'),\\
M_d(x_2,x_2') &= M^2_d(x_2,x_2'),\\
M_d(x_1,x_2) &= \Phi_d(x_1,x_2),\\
M_d(x_2,x_1) &= \emptyset.
\end{split}
\]
The evaluation maps $e_\pm$ on these moduli spaces are the same as the evaluation maps for $\underline{A}_1$, $A_2$, and $\Phi$.
\end{definition}

It follows from the equations \eqref{eqn:conjugatemorphism} that $\underline{A}_1\sqcup_\Phi A_2$ is a Morse-Bott system. We can now use the generic choices $(\mc{P}_1,\mc{P}_2)$ to define the cascade chain complex for this Morse-Bott system. Let $\partial$ denote the differential. Let $\underline{\partial}_{1}$ denote the differential on $C^\ca_*(\underline{A}_1,\mc{P}_1)$, and let $\partial_{2}$ denote the differential on $C^\ca_*(A_2,\mc{P}_2)$. Let $\Phi_\sharp$ denote the portion of $\partial$ mapping from $C^\ca_*(\underline{A}_1,\mc{P}_1)$ to $C^\ca_*(A_2,\mc{P}_2)$, precomposed with the involution $\tau$. We can then write the full cascade differential $\partial$ in block matrix form as
\[
\partial = 
\begin{pmatrix}
\underline{\partial}_{1} & 0 \\ \Phi_\sharp\tau & \partial_{2}
\end{pmatrix}.
\]

Since $\partial^2=0$, it follows that
\[
\partial_{2}\Phi_\sharp\tau + \Phi_\sharp\tau\underline{\partial}_{1} = 0.
\]
Since $\tau\underline{\partial}_{1} = -\partial_{1}\tau$, the above equation is equivalent to
\[
\partial_{2}\Phi_\sharp = \Phi_\sharp\partial_{1}.
\]
Thus $\Phi_\sharp$ is a chain map.

\begin{definition}
\label{def:phistar}
Let
\[
\Phi_*^{\mc{P}_2,\mc{P}_1} : H^\ca_*(A_1,\mc{P}_1) \longrightarrow H^\ca_*(A_2,\mc{P}_2).
\]
denote the map on cascade homology induced by the chain map $\Phi_\sharp$.
\end{definition}

\subsection{More conjugation}

Our next goal is to prove that the induced maps on cascade homology are functorial. To prepare for this, it will be useful to consider the conjugate of a morphism.

\begin{definition}
If $\Phi$ is a morphism of Morse-Bott systems from $A_1$ to $A_2$, define its {\em conjugate\/} $\underline{\Phi}$ by
\[
\underline{\Phi}_d(x_1,x_2) = (-1)^d\Phi_d(x_1,x_2).
\]
Observe that $\underline{\Phi}$ is a morphism of Morse-Bott systems from $A_1$ to $\underline{A}_2$, because the equations \eqref{eqn:conjugatemorphism} still hold if we replace $\Phi$ by $\underline{\Phi}$, $\underline{M}^1$ by $M^1$, and $M^2$ by $\underline{M}^2$.
\end{definition}

\begin{lemma}
\label{lem:conjminus}
The following diagram commutes (note the minus sign):
\[
\begin{CD}
C^\ca_*(\underline{A}_1,\mc{P}_1) @>{-\underline{\Phi}_\sharp}>> C^\ca_*(\underline{A}_2,\mc{P}_2)\\
@V{\tau}VV @VV{\tau}V \\
C^\ca_*(A_1,\mc{P}_1) @>{\Phi_\sharp}>> C^\ca_*(A_2,\mc{P}_2). 
\end{CD}
\]
\end{lemma}

\begin{proof}
The map $\underline{\Phi}_\sharp$ is defined from the cascade differential for the conjugate of the Morse-Bott system $\underline{A}_1\sqcup_\Phi A_2$. The lemma then follows from Lemma~\ref{lem:tauacm} applied to $A=\underline{A}_1\sqcup_\Phi A_2$.
\end{proof}

\subsection{Functoriality}

We are now ready to prove the following key result.

\begin{proposition}
\label{prop:functoriality}
Let $\Phi:A_1\to A_2$ and $\Psi:A_2\to A_3$ be composable morphisms of Morse-Bott systems. Let $(\mc{P}_1,\mc{P}_2,\mc{P}_3)$ be generic choices as needed to define the chain complexes $C^\ca_*(A_i,\mc{P}_i)$. Then
\[
(\Psi\circ\Phi)_*^{\mc{P}_3,\mc{P}_1} = \Psi^{\mc{P}_3,\mc{P}_2}_*\circ \Phi^{\mc{P}_2,\mc{P}_1}_*: H^\ca_*(A_1,\mc{P}_1) \longrightarrow H^\ca_*(A_3,\mc{P}_3).
\]
\end{proposition}

The idea of the proof is to define a chain homotopy between $(\Psi\circ\Phi)_\sharp$ and $\Psi_\sharp\circ\Phi_\sharp$, by counting ``hybrid'' cascades that consist of some elements of $M^1_d$, followed by an element of $\Phi_d$, followed by some elements of $M^2_d$, followed by an element of $\Psi_d$, followed by some elements of $M^3_d$, with total moduli space dimension zero. The chain homotopy equation then comes from considering such cascades with total moduli space dimension one. We will again use a shortcut to simplify the notation.

\begin{proof}[Proof of Proposition~\ref{prop:functoriality}.]
We define an ``almost'' Morse-Bott system
\[
A = A_1\sqcup_{\underline{\Phi}} \underline{A}_2 \sqcup_\Psi A_3 = (X,|\cdot|,S,\mc{O},M_*,e_\pm)
\]
as follows. This will satisfy all of the axioms for a Morse-Bott system, except for a partial failure of the Compactness axiom.

We take $X=X_1\sqcup X_2\sqcup X_3$.

For $x_1\in X_1$ we define $|x_1|=|x_1|_1$, $S_{x_1}=S^1_{x_1}$, and $\mc{O}_{x_1}=\mc{O}^1_{x_1}$. For $x_2\in X_2$ we take $|x_2|=|x_2|_2+1$, $S_{x_2}=S^2_{x_2}$, and $\mc{O}_{x_2}=\mc{O}^2_{x_2}$. For $x_3\in X_3$ we take $|x_3|=|x_3|_3$, $S_{x_3}=S^3_{x_3}$, and $\mc{O}_{x_3} = \mc{O}^3_{x_3}$.

For $x_1,x_1'\in X_1$, $x_2,x_2'\in X_2$, and $x_3,x_3'\in X_3$, we define \[
\begin{split}
M_d(x_1,x_1') &= M^1_d(x_1,x_1'),\\
M_d(x_2,x_2') &= \underline{M}^2_d(x_2,x_2'),\\
M_d(x_3,x_3') &= M^3_d(x_3,x_3'),\\
M_d(x_1,x_2) &=\underline{\Phi}_d(x_1,x_2),\\
M_d(x_2,x_3) &=\Psi_d(x_2,x_3),\\ 
M_d(x_1,x_3)=M_d(x_2,x_1) &= M_d(x_3,x_1) = M_d(x_3,x_2)=\emptyset.
\end{split}
\]
The evaluation maps $e_\pm$ on these moduli spaces are the same as the evaluation maps for $A_1$, $\underline{A}_2$, $A_3$, $\underline{\Phi}$, and $\Psi$.

Observe that $A$ satisfies all of the axioms for a Morse-Bott system, except that parts (b)--(d) of the Compactness axiom fail when applied to $x_1\in X_1$ and $x_3\in X_3$. Namely, parts (b)--(d) of Compactness require that we have
\begin{equation}
\tag{?!}
\partial\overline{M}_1(x_1,x_3)
=
\coprod_{\substack{x_2\in X_2 \\ d_1+d_2=1}}
(-1)^{d_1}
\underline{\Phi}_{d_1}(x_1,x_2)\times_{S(x_2)}\Psi_{d_2}(x_2,x_3),
\end{equation}
and similar equations with generic point constraints on $S_{x_1}$ and/or $S_{x_3}$. In fact, however, the left hand side of each of these equations is empty, since all moduli spaces from $x_1$ to $x_3$ are empty.

What happens if we try to define a cascade ``differential'' $\partial$ for $(A,(\mc{P}_1,\mc{P}_2,\mc{P}_3))$ anyway, despite the above failure of compactness? Proposition~\ref{prop:cascadekey}(a) still holds, so we obtain a well-defined linear map
\[
\partial: C^\ca_*(A,(\mc{P}_1,\mc{P}_2,\mc{P}_3)) \longrightarrow C^\ca_{*-1}(A,(\mc{P}_1,\mc{P}_2,\mc{P}_3)).
\]
However we no longer know that $\partial^2=0$. In particular Proposition~\ref{prop:cascadekey}(b) no longer holds when $x_+=x_1\in X_1$ and $x_-=x_3\in X_3$. We will need to compute the precise error in order to find out what $\partial^2$ actually is.

Fix $\widetilde{x}_1$ to denote $\widecheck{x}_1$ or $\widehat{x}_1$, and fix $\widetilde{x}_3$ to denote $\widecheck{x}_3$ or $\widehat{x}_3$. The part of the proof of Proposition~\ref{prop:cascadekey}(b) that is no longer valid is equation \eqref{eqn:bpc}, which in this case would be
\begin{equation}
\tag{?!}
\begin{split}
\partial\widetilde{M}^\ca_1(\widetilde{x}_1,\widetilde{x}_3) =& \coprod_{x'\neq x_1,x_3} M_0^\ca(\widetilde{x}_1,\widecheck{x}')\times_{S(x')} M^\ca_1(\widehat{x}',\widetilde{x}_3)\\
& - \coprod_{x'\neq x_1,x_3} M_1^\ca(\widetilde{x}_1,\widecheck{x}') \times_{S(x')} M^\ca_0(\widehat{x}',\widetilde{x}_3).
\end{split}
\end{equation}
In the present case, the left hand side is missing points on the right hand side in which $x'=x_2\in X_2$ and the cascades do not involve any other elements of $X_2$. More precisely, define
\begin{equation}
\label{eqn:defZ}
\begin{split}
Z(\widetilde{x}_1,\widetilde{x}_3) = & \coprod_{x_2\in X_2} \widehat{M}_0^\ca(\widetilde{x}_1,\widecheck{x}_2)\times_{S(x_2)} \widehat{M}^\ca_1(\widehat{x}_2,\widetilde{x}_3)
\\
&- \coprod_{x_2\in X_2} \widehat{M}_1^\ca(\widetilde{x}_1,\widecheck{x}_2) \times_{S(x_2)} \widehat{M}^\ca_0(\widehat{x}_2,\widetilde{x}_3),
\end{split}
\end{equation}
where we use the notation $\widehat{M}^\ca$ to indicate cascades which involve only one point in $X_2$. We can then correct the previous equation by adding $Z(\widetilde{x}_1,\widetilde{x}_3)$ to the left hand side, giving
\begin{equation}
\label{eqn:bandaid}
\begin{split}
\partial\widetilde{M}^\ca_1(\widetilde{x}_1,\widetilde{x}_3) \sqcup Z(\widetilde{x}_1,\widetilde{x}_3) =& \coprod_{x'\neq x_1,x_3} M_0^\ca(\widetilde{x}_1,\widecheck{x}')\times_{S(x')} M^\ca_1(\widehat{x}',\widetilde{x}_3)\\
& - \coprod_{x'\neq x_1,x_3} M_1^\ca(\widetilde{x}_1,\widecheck{x}') \times_{S(x')} M^\ca_0(\widehat{x}',\widetilde{x}_3).
\end{split}
\end{equation}
The rest of the proof of Proposition~\ref{prop:cascadekey}(b) now goes through. However since the right hand side of \eqref{eqn:bandaid} is used to cancel some ends of the moduli space, and since we had to add the points in $Z(\widetilde{x}_1,\widetilde{x}_3)$ to obtain this right hand side, the result is that we obtain a compactification $\overline{M}^\ca_1(\widetilde{x}_1,\widetilde{x}_3)$ of $M^\ca_1(\widetilde{x}_1,\widetilde{x}_3)$ with $-Z(\widetilde{x}_1,\widetilde{x}_3)$ added to its boundary. That is,
\begin{equation}
\label{eqn:ccc}
\begin{split}
\partial \overline{M}^\ca_1(\widetilde{x}_1,\widetilde{x}_3) = & \coprod_{y\in X}M^\ca_0(\widetilde{x}_1,\widecheck{y}) \times M^\ca_0(\widecheck{y},\widetilde{x}_3)\\
&\bigsqcup \coprod_{y\in X} M^\ca_0(\widetilde{x}_1,\widehat{y}) \times M^\ca_0(\widehat{y},\widetilde{x}_3)\\
&-Z(\widetilde{x}_1,\widetilde{x}_3).
\end{split}
\end{equation}

Let $\widehat{M}^\ca_0(\widetilde{x}_1,\widetilde{x}_3)$ denote the moduli space of cascades in $\underline{A}_1\sqcup_{\Psi\circ\Phi} A_3$ from $\widetilde{x}_1$ to $\widetilde{x}_3$ in which each factor lives in a zero dimensional moduli space. Observe that as a set, we have $Z(\widetilde{x}_1,\widetilde{x}_3) = \widehat{M}^\ca_0(\widetilde{x}_1,\widetilde{x}_3)$.
We claim that as an oriented $0$-manifold, we have
\begin{equation}
\label{eqn:Zor}
Z(\widetilde{x}_1,\widetilde{x}_3)
 =
\left\{
\begin{array}{cl}
-\widehat{M}^\ca_0(\widetilde{x}_1,\widetilde{x}_3), & \widetilde{x}_1=\widecheck{x}_1,\\
\widehat{M}^\ca_0(\widetilde{x}_1,\widetilde{x}_3), & \widetilde{x}_1 = \widehat{x}_1.
\end{array}
\right.
\end{equation}
To see how the orientations work, consider a cascade in
\[
(u_0,\ldots,u_k) \in \widehat{M}^\ca_0(\widetilde{x}_1,\widetilde{x}_3).
\]
Here $u_j\in (\Psi\circ\Phi)_0$ for some $j\in\{0,\ldots,k\}$; $u_i\in\underline{M}^1$ for all $i<j$; and $u_i\in M^3$ for all $i>j$.
Assume for simplicity that $u_j$ does not have any point constraints. Then all factors $u_i$ for $i>0$ are counted with the same signs in $\widehat{M}^\ca_0(\widetilde{x}_1,\widetilde{x}_3)$ and in $Z(\widetilde{x}_1,\widetilde{x}_3)$; the minus sign in the second line of \eqref{eqn:defZ} arises because cascades in $\widehat{M}^\ca_1(\widetilde{x}_1,\widehat{x}_2)$ are oriented using $\underline{\Phi}$ instead of $\Phi$. The factor $u_0$ is counted with the same sign in $\widehat{M}^\ca_0(\widetilde{x}_1,\widetilde{x}_3)$ and $Z(\widetilde{x}_1,\widetilde{x}_3)$ when $\widetilde{x}_1=\widehat{x}_1$, and with opposite signs when $\widetilde{x}_1=\widecheck{x}_1$. The reason is that in the latter case, $u_0\in \underline{M}^1_1=-M^1_1$. The cases where $u_j$ has point constraints are treated similarly. 

Combining \eqref{eqn:Zor} with \eqref{eqn:ccc}, it follows that the part of $\partial^2$ mapping from $C_*^\ca(A_1,\mc{P}_1)$ to $C_*^\ca(A_3,\mc{P}_3)$ is given by
\begin{equation}
\label{eqn:d2failure}
\partial^2 = -(\Psi\circ\Phi)_\sharp.
\end{equation}
Now let $K$ denote the portion of $\partial$ mapping from $C^\ca_*(A_1,\mc{P}_1)$ to $C^\ca_*(A_3,\mc{P}_3)$. We can then write $\partial_\ca$ in block matrix form as
\[
\partial = 
\begin{pmatrix}
\partial_{1} & 0 & 0\\
\underline{\Phi}_\sharp\tau & \underline{\partial}_{2} & 0\\
K & \Psi_\sharp\tau & \partial_{3}
\end{pmatrix}.
\]
Squaring this and comparing the lower left entry with \eqref{eqn:d2failure}, we obtain
\[
K\partial_{1} + \Psi_\sharp\tau\underline{\Phi}_\sharp\tau + \partial_{3}K = -(\Psi\circ\Phi)_\sharp.
\]
By Lemma~\ref{lem:conjminus}, we can rewrite this as
\[
K\partial_{1} + \partial_{3}K = \Psi_\sharp\circ\Phi_\sharp - (\Psi\circ\Phi)_\sharp.
\]
Thus $K$ is a chain homotopy between $\Psi_\sharp\circ\Phi_\sharp$ and $(\Psi\circ\Phi)_\sharp$.
\end{proof}

\subsection{Homotopy invariance}

Let $A_1=(X_1,|\cdot|_1,S_1,\mc{O}^1,M^1_*,e^1_\pm)$ and $A_2=(X_2,|\cdot|_2,S_2,\mc{O}^2,M^2_*,e^2_\pm)$ be Morse-Bott systems. We now prove:

\begin{proposition}
\label{prop:homotopyinvariance}
Let $\Phi,\Phi':A_1\to A_2$ be morphisms of Morse-Bott systems. Suppose there exists a homotopy $K$ from $\Phi$ to $\Phi'$ as in Definition~\ref{def:homotopy}. Let $(\mc{P}_1,\mc{P}_2)$ be a generic pair of choices as needed to define the cascade chain complexes $C^\ca_*(A_1,\mc{P}_1)$ and $C^\ca_*(A_2,\mc{P}_2)$. Then
\[
\Phi_* = (\Phi')_*: H_*^\ca(A_1,\mc{P}_1) \longrightarrow H_*^\ca(A_2,\mc{P}_2).
\]
\end{proposition}

To prove Proposition~\ref{prop:homotopyinvariance}, we define a chain homotopy
\[
K_\sharp: C^\ca_*(A_1,\mc{P}_1) \longrightarrow C^\ca_{*+1}(A_2,\mc{P}_2).
\]
from $\Phi_\sharp$ to ${\Phi'}_\sharp$. To do so, we define an ``almost'' Morse-Bott system
\[
A_1\sqcup_K A_2 = (X,|\cdot|,S,\mc{O},M_*,e_\pm),
\]
similarly to Definition~\ref{def:A1A2}, as follows. We take $X = X_1\sqcup X_2$. For $x_1\in X_1$ we define $|x_1|=|x_1|_1+2$, $S_{x_1}=S^1_{x_1}$, and $\mc{O}_{x_1}=\mc{O}_{x_1}^1$. For $x_2\in X_2$ we define $|x_2|=|x_2|_2$, $S_{x_2}=S^2_{x_2}$, and $\mc{O}_{x_2}=\mc{O}^2_{x_2}$. For $x_1,x_1'\in X_1$ and $x_2,x_2'\in X_2$, we define
\[
\begin{split}
M_d(x_1,x_1') &= M^1_d(x_1,x_1'),\\
M_d(x_2,x_2') &= M_d^2(x_2,x_2'),\\
M_d(x_1,x_2) &= K_d(x_1,x_2),\\
M_d(x_2,x_1) &= \emptyset.
\end{split}
\]
The evaluation maps $e_\pm$ on these moduli spaces are the same as the evaluation maps for $A_1$, $A_2$, and $K$.

As in Definition~\ref{def:A1A2}, $A_1\sqcup_K A_2$ almost satisfies the axioms for a Morse-Bott system, except that we do not have parts (b), (c), and (d) of the Compactification axiom, because of the extra terms involving $\Phi$ and $\Phi'$ in equations \eqref{eqn:compactopy1}, \eqref{eqn:compactopy2-}, \eqref{eqn:compactopy2+}, and \eqref{eqn:compactopy3}. In any case, since part (a) of Compactification holds, it still makes sense to define a cascade ``differential'' $\partial$. We write this in block matrix form as
\begin{equation}
\label{eqn:defKsharp}
\partial = \begin{pmatrix} \partial_{1} & 0 \\ K_\sharp & \partial_{2} \end{pmatrix},
\end{equation}
and this is the definition of $K_\sharp$.

\begin{lemma}
$K_\sharp$ satisfies the chain homotopy equation
\[
\partial_{2}K_\sharp + K_\sharp\partial_{1} = \Phi'_\sharp - \Phi_\sharp.
\]
\end{lemma}

\begin{proof}
Since we do not have parts (b)--(d) of the Compactness axiom, we do not have $\partial^2=0$. Instead, taking note of the extra terms involving $\Phi$ and $\Phi'$ in equations \eqref{eqn:compactopy1}, \eqref{eqn:compactopy2-}, \eqref{eqn:compactopy2+}, and \eqref{eqn:compactopy3}, and comparing with the definition of $\Phi_\sharp$ and $\Phi'_\sharp$, we find that
\begin{equation}
\label{eqn:Kboundary}
\partial^2 = \begin{pmatrix} 0 & 0 \\ \Phi'_\sharp - \Phi_\sharp & 0 \end{pmatrix}.
\end{equation}

To explain the signs in this equation, consider a cascade contributing to the coefficient $\langle\partial^2\widetilde{x}_+,\widetilde{x}_-\rangle$, coming from a $\Phi$ boundary term in equation \eqref{eqn:compactopy1}, \eqref{eqn:compactopy2-}, \eqref{eqn:compactopy2+}, or \eqref{eqn:compactopy3}. At first glance, these equations suggest that this cascade should count with the same sign as in $\Phi_\sharp$ when $\widetilde{x}_+=\widehat{x}$ (which would disagree with \eqref{eqn:Kboundary}), and with the opposite sign as in $\Phi_\sharp$ when $\widetilde{x}_+=\widecheck{x}_-$ (which would agree with \eqref{eqn:Kboundary}). However we have to make two adjustments in order to compare with the signs in \eqref{eqn:compactopy1}--\eqref{eqn:compactopy3} with the signs in the definition of $\Phi_\sharp$: namely, we have to replace $A_1$ by $\underline{A}_1$ and insert $\tau$.

If $\widetilde{x}_+=\widehat{x}_+$, then replacing $A_1$ by $\underline{A}_1$ does not affect the sign, because any $M^1$ factors in the cascade are in zero-dimensional moduli spaces. However the $\tau$ factor in the definition of $\Phi_\sharp$ does switch the sign.

On the other hand, if $\widetilde{x}_+=\widecheck{x}_-$, then there are two cases to consider. If the first factor in the cascade is in $\Phi$, then replacing $A_1$ by $\underline{A}_1$ does not affect the sign, and inserting $\tau$ does not affect the sign either. If the first factor in the cascade is in $M^1$, then it lives in a one-dimensional moduli space, while all other factors in $M^1$ live in zero-dimensional moduli spaces. Thus replacing $A_1$ by $\underline{A}_1$ switches the sign; and inserting $\tau$ also switches the sign. This completes the proof of \eqref{eqn:Kboundary}.

Computing $\partial^2$ using \eqref{eqn:defKsharp} and comparing with \eqref{eqn:Kboundary}, we obtain
\[
\partial_2K_\sharp + K_\sharp\underline{\partial}_1 = \Phi'_\sharp - \Phi_\sharp.
\]
\end{proof}

\subsection{Independence of the choice of base points}
\label{sec:ibp}

We now show that if $A$ is a Morse-Bott system, then the cascade homology $H_*^\ca(A,\mc{P})$ does not depend on the choice of base points $\mc{P}$, and so can be denoted by $H_*^\ca(A)$. In addition, if $\Phi$ is a morphism of Morse-Bott systems from $A_1$ to $A_2$, then the induced map on cascade homology $\Phi_*^{\mc{P}_2,\mc{P}_1}:H_*^\ca(A_1,\mc{P}_1) \to H_*^\ca(A_2,\mc{P}_2)$ gives a well-defined map $\Phi_*:H_*^\ca(A_1)\to H_*^\ca(A_2)$. More precisely:

\begin{proposition}
\label{prop:basepoint}
Let $A$ be a Morse-Bott system, and let $\mc{P},\mc{P}'$ be two choices of base points as needed to define the cascade chain complex. Then there is a canonical isomorphism
\begin{equation}
\label{eqn:phipp'}
\phi_{\mc{P}',\mc{P}}: H_*^\ca(A,\mc{P}) \stackrel{\simeq}{\longrightarrow} H_*^\ca(A,\mc{P}')
\end{equation}
with the following properties:
\begin{description}
\item{(a)}
$\phi_{\mc{P},\mc{P}} = \op{id}_{H_*^\ca(A,\mc{P})}$.
\item{(b)}
If $\mc{P}''$ is a third choice of base points, then
\[
\phi_{\mc{P}'',\mc{P}} = \phi_{\mc{P}'',\mc{P}'}\circ \phi_{\mc{P}',\mc{P}}: H_*^\ca(A,\mc{P}) \longrightarrow H_*^\ca(A,\mc{P}'').
\]
\item{(c)}
If $\Phi$ is a morphism of Morse-Bott systems from $A_1$ and $A_2$, if $\mc{P}_1$ and $\mc{P}_1'$ are choices of base points for $A_1$, and if $\mc{P}_2$ and $\mc{P}_2'$ are choices of base points for $A_2$, then
the following diagram commutes:
\[
\begin{CD}
H_*^\ca(A_1,\mc{P}_1) @>{\phi_{\mc{P}_1',\mc{P}_1}}>> H_*^\ca(A_1,\mc{P}_1')\\
@V{\Phi_*^{\mc{P}_2,\mc{P}_1}}VV @VV{\Phi_*^{\mc{P}'_2,\mc{P}'_1}}V\\
H_*^\ca(A_2,\mc{P}_2) @>{\phi_{\mc{P}_2',\mc{P}_2}}>> H_*^\ca(A_2,\mc{P}_2')\\ 
\end{CD}
\]
\end{description}
\end{proposition}

\begin{proof}
To define the map \eqref{eqn:phipp'}, suppose that the pair $(\mc{P},\mc{P}')$ is generic. Then by \S\ref{sec:inducedmap}, the identity morphism of $A$ induces a map
\[
\op{id}_*^{\mc{P}',\mc{P}} : H_*^\ca(A,\mc{P})\longrightarrow H_*^\ca(A,\mc{P}'),
\]
and we define this to be $\phi_{\mc{P}',\mc{P}}$.

\begin{lemma}
\label{lem:phiso}
The map $\phi_{\mc{P}',\mc{P}}$ defined above for generic pairs $(\mc{P},\mc{P}')$ is an isomorphism.
\end{lemma}

\begin{proof}
Let $C_*$ denote the free $\Z$-module with generators $\widecheck{x}$ and $\widehat{x}$ for each $x\in X$. Choose generators of $\mc{O}_x(p_x)$ and $\mc{O}_x(p_x')$ for each $x\in X$, in order to identify both cascade chain modules $C_*^\ca(A,\mc{P})$ and $C_*^\ca(A,\mc{P}')$ with $C_*$. For each $x\in X$, choose these generators to agree under the isomorphism $\mc{O}_x(p_x)\stackrel{\simeq}{\to}\mc{O}_x(p_x')$ given by parallel transport along a positively oriented embedded arc on $S(x)$ from $p_x$ to $p_x'$. Then by the construction in \S\ref{sec:inducedmap}, $\phi_{\mc{P}',\mc{P}}$ is induced by a chain map of the form
\[
I + B : C_* \longrightarrow C_*
\]
where $I$ denotes the identity on $C_*$, and $\langle B\widetilde{x},\widetilde{y}\rangle\neq 0$ only if $x\neq y$ and $M_d(x,y)\neq\emptyset$ for some $d\in\{0,1\}$. Now
\[
\sum_{k=0}^\infty(-1)^kB^k: C_*\longrightarrow C_*
\]
is a well-defined linear map, by the Finiteness axiom for a Morse-Bott system, and it is inverse to $I+B$. Thus $I+B$ is an isomorphism of chain complexes, and so it induces an isomorphism on homology.
\end{proof}

The above definition of $\phi_{\mc{P}',\mc{P}}$ only works for generic pairs $(\mc{P},\mc{P}')$; in particular it does not work when $\mc{P}=\mc{P}'$. To extend the definition to arbitrary pairs $(\mc{P},\mc{P}')$ for which both cascade chain complexes are defined, we use the following lemma:

\begin{lemma}
\label{lem:bp}
Let $\Phi$ be a morphism of Morse-Bott systems from $A_1$ to $A_2$, and let $\mc{P}_1$ and $\mc{P}_2'$ be generic choices of base points for $A_1$ and $A_2$, so that the cascade chain complexes and the map $\Phi_*^{\mc{P}_2',\mc{P}_1}$ are defined. Then:
\begin{description}
\item{(a)}
If $\mc{P}_1'$ is a generic choice of base points for $A_1$, so that $\phi_{\mc{P}_1',\mc{P}_1}$ and $\Phi_*^{\mc{P}_2',\mc{P}_1'}$ are defined, then
\[
\Phi_*^{\mc{P}_2',\mc{P}_1} = \Phi_*^{\mc{P}_2',\mc{P}_1'}\circ\phi_{\mc{P}_1',\mc{P}_1}.
\]
\item{(b)}
If $\mc{P}_2$ is a generic choice of base points for $A_2$, so that $\Phi_*^{\mc{P}_2,\mc{P}_1}$ and $\phi_{\mc{P}_2',\mc{P}_2}$ are defined, then
\[
\Phi_*^{\mc{P}_2',\mc{P}_1} = \phi_{\mc{P}_2',\mc{P}_2}\circ \Phi_*^{\mc{P}_2,\mc{P}_1}.
\]
\end{description}
\end{lemma}

\begin{proof}
The idea is to apply the functoriality of Proposition~\ref{prop:functoriality} to the composition of $\Phi$ with the identity morphism for $A_1$ or $A_2$. Unfortunately we cannot do this directly, because as discussed in Remark~\ref{rem:identity}, $\Phi$ might not be composable with the identity; and even when it is, the composition of $\Phi$ with the identity is different from $\Phi$ (although only inconsequentially). However the proof of Proposition~\ref{prop:functoriality} still goes through in this case with minor modifications. We omit the details.
\end{proof}

\begin{lemma}
\label{lem:phiple}
Let $(\mc{P}_1,\mc{P}_2,\mc{P}_3)$ be a generic triple of choices of base points for the Morse-Bott system $A$. (If any pair of these choices is generic, then the third choice can be made generically.) Then
\[
\phi_{\mc{P}_3,\mc{P}_1} = \phi_{\mc{P}_3,\mc{P}_2}\circ \phi_{\mc{P}_2,\mc{P}_1}.
\]
\end{lemma}

\begin{proof}
This is a special case of Lemma~\ref{lem:bp} in which $\Phi:A\to A$ is the identity morphism.
\end{proof}

Continuing the proof of Proposition~\ref{prop:basepoint}, if $A$ is a Morse-Bott system and if $\mc{P},\mc{P}'$ are any choices of base points for which the cascade chain complexes are defined, then we can define $\phi_{\mc{P}',\mc{P}}$ by generically choosing a third set of base points $\mc{P}''$ and setting
\[
\phi_{\mc{P}',\mc{P}} = \phi_{\mc{P}',\mc{P}''}\circ \phi_{\mc{P}'',\mc{P}}.
\]
It follows by repeatedly applying Lemma~\ref{lem:phiple} that $\phi_{\mc{P}',\mc{P}}$ does not depend on the choice of $\mc{P}''$ and satisfies property (b) in Proposition~\ref{prop:basepoint}. It follows from Lemma~\ref{lem:phiso} that $\phi_{\mc{P}',\mc{P}}$ is an isomorphism.

To prove part (a) of Proposition~\ref{prop:basepoint}, by definition we have
\[
\phi_{\mc{P},\mc{P}} = \phi_{\mc{P},\mc{P}'}\circ\phi_{\mc{P}',\mc{P}}
\]
where $\mc{P}'$ is generic. To prove that $\phi_{\mc{P},\mc{P}}$ is the identity, by Lemma~\ref{lem:phiso} it is enough to show that
\[
\phi_{\mc{P},\mc{P}'}\circ\phi_{\mc{P}',\mc{P}}\circ\phi_{\mc{P},\mc{P}''} = \phi_{\mc{P},\mc{P}''}
\]
where $\mc{P}''$ is generic. This last equation follows by applying Lemma~\ref{lem:phiple} twice.

Part (c) of Proposition~\ref{prop:basepoint} now follows from Lemma~\ref{lem:bp}.
\end{proof}

\subsection{Proof of the main theorem}
\label{sec:conclusion}

To conclude, we now review how the above results prove all of the points in the main theorem.

\begin{proof}[Proof of Theorem~\ref{thm:main}.]
Part (a) follows from Proposition~\ref{prop:basepoint}(a),(b).

Part (b.i) follows from Proposition~\ref{prop:basepoint}(c).

Part (b.ii) holds by definition, because we are using the maps induced by the identity morphism to identify the cascade homologies for different choices of base points with each other.

Part (b.iii) follows from Proposition~\ref{prop:functoriality}.

Part (b.iv) follows from Proposition~\ref{prop:homotopyinvariance}.
\end{proof}

\noindent \textsc{Michael Hutchings \\  University of California at Berkeley}\\
{\em email: }\texttt{hutching@math.berkeley.edu}\\

\noindent \textsc{Jo Nelson \\  Rice University}\\
{\em email: }\texttt{jo.nelson@rice.edu}\\

\end{document}